\documentclass[a4paper,12]{article}
\usepackage{amsmath, amsfonts, amssymb, amsthm, bm, graphics, bbm, color}
\usepackage{graphicx}
\usepackage{subfigure}
\usepackage{mathtools}
\usepackage[table]{xcolor}
\usepackage{url}
\usepackage{mathabx}
\usepackage{cancel}
\usepackage{multicol}
\usepackage{float}
\usepackage{caption}
\usepackage{mathalfa}
\usepackage{hyperref}
\usepackage{afterpage}
\usepackage{algorithm}
\usepackage{algpseudocode}
\usepackage{multirow}
\DeclareMathAlphabet\mathbfcal{OMS}{cmsy}{b}{n}
\newtheorem{theorem}{Theorem}[section]
\newtheorem{lemma}[theorem]{Lemma}

\newtheorem{proposition}[theorem]{Proposition}
\theoremstyle{definition}

\newtheorem{remark}[theorem]{Remark}

\newcolumntype{P}[1]{>{\centering\arraybackslash}p{#1}}
 \DeclareMathOperator{\Tr}{tr}

\makeatletter
\newcommand\makebig[2]{%
  \@xp\newcommand\@xp*\csname#1\endcsname{\bBigg@{#2}}%
  \@xp\newcommand\@xp*\csname#1l\endcsname{\@xp\mathopen\csname#1\endcsname}%
  \@xp\newcommand\@xp*\csname#1r\endcsname{\@xp\mathclose\csname#1\endcsname}%
}
\makeatother

\makebig{biggg} {3.0}
\makebig{Biggg} {3.5}
\makebig{bigggg}{4.0}
\makebig{Bigggg}{4.5}
\makebig{biggggg}{5.0}
\makebig{Biggggg}{5.5}
\makebig{bigggggg}{6.0}
\makebig{Bigggggg}{6.5}

\renewcommand{\vec}[1]{\mbox{\boldmath$#1$}}

\newcommand{\dif}{\mathrm{d}}
\newcommand{\im}{\mathrm{i}}

\newcommand{\IH}{{\text{IH}}}
\newcommand{\MH}{{\text{MH}}}

\begin{document}
\title{Minimal Object Characterisations using Harmonic Generalised Polarizability Tensors and Symmetry Groups}

\author{P.D. Ledger$^\dagger$  and W.R.B. Lionheart$^\ddagger$\\
$^\dagger$School of Computing \& Mathematics, Keele University, \\
Keele, Staffordshire ST5 5BG United Kingdom.\\
$^\ddagger$Department of Mathematics, The University of Manchester,\\
Oxford Road, Manchester M13 9PL, United Kingdom.\\
Corresponding author: p.d.ledger@keele.ac.uk}
\date{Submitted 14th January 2022, Revised 1st August 2022}

\maketitle

\section*{Abstract}
We introduce a new type of  object characterisation, which is capable of accurately describing small isolated inclusions for potential field inverse problems such as in electrostatics, magnetostatics and related low frequency Maxwell problems. Relevant applications include characterising ferrous unexploded ordnance (UXO) from magnetostatic field measurements in magnetometry, describing small conducting inclusions for medical imaging using electrical impedance tomography (EIT), performing geological ground surveys using electrical resistivity imaging (ERT), characterising objects by electrosensing fish to navigate and identify food  as well as describing the effective properties of dilute composites. Our  object characterisation builds on the generalised polarizability tensor (GPT) object  characterisation concept and provides an alternative to the compacted GPT (CGPT). We call the new characterisations  harmonic GPTs (HGPTs) as their coefficients correspond to  products of harmonic polynomials. Then, we show that the number of independent coefficients of HGPTs needed to characterise objects can be significantly reduced by considering the symmetry group of the object and propose a systematic approach
for determining the  
subspace of symmetric harmonic polynomials that is fixed by the group and its dimension. This enable us to determine the independent HGPT coefficients for different symmetry groups.
 
 \noindent{\bf Keywords:} Inverse problems, generalised polarizability tensor, object characterisation, symmetry groups, magnetometry, electrical impedance tomography. 

\noindent {\bf MSC Classification:} 35R30; 35B30; 20C30

\section{Introduction}

The purpose of this paper is to introduce a new type of  object characterisation, which is capable of accurately describing small isolated inclusions for potential field inverse problems such as in electrostatics, magnetostatics and related low frequency Maxwell problems. This is important for magnetometry, which uses variations in the earth's magnetic field caused by the presence of hidden ferrous objects to distinguish between unexploded ordnance (UXOs) and metallic shrapnel as well as identify archaeological features. Further applications include: finding ferrous objects from metal detection measurements at very low frequencies (where only the magnetic part of the characterisation can be recorded);  describing the effective properties of dilute composites and characterising small conducting inclusions for applications in electrical impedance tomography (EIT)~\cite{billeit} and electrical resistivity imaging (ERT)~\cite{billert}
as well as characterising objects by electrosensing fish in order to navigate and identify food~\cite{fishbill}.
 EIT offers possibilities for low-cost non-invasive medical imaging such as in lung monitoring. Here, the electrical conductivity, permittivity, and impedance of a part of the body is inferred from surface electrode measurements and used to form a tomographic image. Related to EIT is ERT, which is a geophysical technique for imaging sub-surface structures from electrical resistivity measurements made at the surface. In a similar way, weakly electric fish generate electric current and use hundreds of voltage sensors on the surface of their body to navigate and locate food.  Experiments have shown that they can discriminate between differently shaped conducting or insulating objects by using electrosensing~\cite{fishexp}.
 
Our object characterisation builds on the generalized polarizability tensor (GPT) object characterisation concept developed by Ammari and Kang~\cite{ammarikangbook} and their coworkers. The simplest form of GPT is a rank 2 (P\'oyla-Szeg\"o) tensor, which describes the shape and material contrast of the object by the best fitting ellipsoid, while the complete set of GPTs uniquely defines both the shape and material contrast of the object~\cite{ammarikangbook}[pg. 90]. However, the additional information provided by higher order GPTs  remains open.
To help to address this, we provide an alternative to their compacted GPT (CGPT) object description in two dimensions~\cite{ammari2012} and three dimensions~\cite{ammari2013shape}, in which perturbed field measurements are expressed as sums of products of  CGPT coefficients and spherical harmonics. We propose an alternative object description called a harmonic GPT (HGPT) where the perturbed field can be described in terms of expansion involving HGPT coefficients and products of harmonic polynomials.
HGPTs have the same dimension of CGPTs and they both provide a significant reduction in the number of coefficients needed to describe an object compared to GPTs of the same degree. 
For objects with rotational or reflectional symmetries, the number of independent coefficients is much smaller in all cases.

By grouping an object according to their symmetry group class, we show that a systematic approach can be developed for determining the  
subspace of symmetric harmonic polynomials that are fixed by the group, and its dimension. This allows us to deduce the HGPT coefficients, of a given degree, which are invariant under the action of the set of orthogonal matrices making up the symmetry group. Then, by considering HGPTs upto a certain order, we can find objects of a certain cyclic (or dihedral) group, hence, contributing to understanding the additional information that higher order (H)GPTs provide. Furthermore, by fixing classes associated with different symmetry groups, the sets of invariant HGPT coefficients offer  alternative features to the shape invariant descriptors based on CGPTs proposed by Ammari, Chung, Kang and Wang~\cite{ammari2013shape}  for dictionary based object classification.
 We also review the related work of Meyer~\cite{meyer} who describes an alternative approach to the one advocated in this paper for determining the subspace of harmonic polynomials that are fixed by a symmetry group. We make the historical note that Burnett Meyer acknowledged in~\cite{meyer} that George P\'olya, his PhD advisor at Stanford, suggested the problem of invariant harmonic polynomials. The first term in the GPT, the rank-2 tensor, was first introduced by P\'olya and Szeg\"o in their 1951 book~\cite{polya}.  Historians of mathematics may be interested to investigate if the connection was accidental or  points to a deeper insight.
 
The paper is organised as follows: In Section~\ref{sect:prelim} we define the mathematical problem that will be our focus  in this paper and collect together some observations about spherical harmonics and harmonic polynomials. Then, in Section~\ref{sect:object}, we review the concepts of GPTs and CGPTs and introduce our new HGPTs. Section~\ref{sect:hgptprop} presents transformation formulae for HGPTs. Next, in Section~\ref{sect:sym}, we describe how  knowledge of the symmetry group of an object can be used  to determine the  
subspace of symmetric harmonic polynomials that is fixed by the group and its dimension. This, in turn,  allows us to determine the independent coefficients of HGPTs for objects associated with different symmetry groups. In this section, we also review the related work of Meyer on  determining the subspace of harmonic polynomials that are fixed by a symmetry group as well as providing tables of symmetric harmonic polynomials fixed by different groups. Finally, some examples of our approach are included for different groups.

\section{Preliminaries}\label{sect:prelim}
\subsection{Problem Definition} \label{sect:define}
The problem of interest in this work is that described in Section 4.1 of~\cite{ammarikangbook}, which we briefly summarise below. We let $B$ be a bounded Lipschitz domain in ${\mathbb R}^3$, and the material contrast of $B$  be $k$. In the case of magnetostatics, and in magnetometry, $k$ can be understood as a real valued contrast involving the magnetic susceptibility and permeability, while in electrostatics, and in related low frequency Maxwell problems such as in EIT~\cite{billeit}, ERT~\cite{billert} and electrosensing~\cite{fishbill}, $k$ is a complex contrast involving the  permittivity, frequency and conductivity. For simplicity, we consider the case of real valued $k$, with $0< k \ne 1< + \infty$, and use $\lambda: = (k+1)/(2 (k-1))$
 in the following. We suppose that the origin $O \in B$ and let $U$ be a harmonic (background) field in ${\mathbb R}^3$ and let $u$ be the solution to
\begin{subequations} \label{eqn:model}
\begin{align}
\nabla \cdot ( (  (k-1) \chi(B))\nabla u) & =0 &&\text{in ${\mathbb R}^3$}, \\
(u - U)({\vec x}) & =O(|{\vec x}|^{-2}) && \text{as $|{\vec x}| \to \infty$},
\end{align}
\end{subequations}
where $\chi(B)$ is $1$ in $B$ and $0$ outside. Our interest lies in describing $(u-U)({\vec x}) $ for the purpose of object characterisations. In magnetostatics, $\nabla_{\vec x} (u-U)({\vec x}) $ corresponds to the perturbation in magnetic field caused by the presence of the inclusion, while in electrostatics,  $\nabla_{\vec x} (u-U)({\vec x}) $ corresponds to the perturbed electric field.

\subsection{Spherical harmonics and harmonic polynomials}\label{sect:spharmoics}

This section summarise some key results about spherical harmonics and harmonic polynomials  that are relevant for what follows. For further details see~\cite{ammari2013shape,nedelecbook}. Given a direction $(\theta,\psi)$ in spherical coordinates, 
the (complex) spherical harmonics of homogeneous degree $n$ and order $m$, with $-n \le m \le n$, are given by
\begin{equation}
Y_n^m(\theta, \psi) = (-1)^m \left [
\frac{2n+1}{4 \pi} \frac{(n-m)!}{(n+m)! } \right ]^{1/2} e^{\im m \psi  } P_n^m( \cos \theta),
\end{equation}
where $P_n^m$ are the associated Legendre polynomials of degree $n$ and order $m$.  The result
\begin{equation}
P_n( \cos \gamma )  = \frac{4 \pi }{2 n+1} \sum_{m=-n}^n Y_n^m(\theta,\phi) \overline{  Y_n^m(\theta',\phi') } , \label{eqn:addsph}
\end{equation}
is known as the addition formula, where $\cos \gamma = \cos \theta \cos \theta'+ \sin \theta \sin \theta' \cos (\phi - \phi')$,  $P_n(x)$ are the Legendre polynomials of degree $n$ and the overbar denotes the complex conjugate. Note that $P_n^m(x)$  is related to $P_n(x)$ by $P_n^m (x) = (-1)^m (1-x^2)^{m/2} \frac{\dif^m}{\dif x^m} ( P_n(x))$.
It is well known that 
\begin{equation}
H_n^m({\vec x}) = r^n Y_n^m(\theta, \psi) ,
\end{equation}
are homogenous harmonic functions where $(r,\theta,\psi)$, with $r=|{\vec x}|$, is the description of ${\vec x}$ in spherical coordinates. As well as being harmonic, these functions are smooth at the origin and tend to infinity at infinity. Interestingly, the functions
\begin{equation}
K_n^m({\vec x}) = \frac{1}{r^{n+1}} Y_n^m(\theta, \psi)  = \frac{1}{r^{2n+1}} H_n^m({\vec x}),
\end{equation}
are also harmonic, but are discontinuous at the origin and tend to zero at infinity~\cite{nedelecbook}[pg. 40]. The $2n+1$ harmonic functions $H_n^m({\vec x})$ of degree $n$ can be expressed in terms of a basis of real valued harmonic polynomials $I_n^{\ell}  ({\vec x})$ using
\begin{equation}
H_n^m({\vec x}) = \sum_{\ell = -n}^{n} a_{\ell m}^{\IH}  I_{n}^\ell ({\vec x}), \label{eqn:himpoly}
\end{equation}
which has $2n +1$ terms, an expansion that is smaller than the dimension {$(n+1)(n+2)/2$} of the standard monomial expansion $\sum_{\beta, |\beta | =n} a_\beta {\vec x}^\beta$ of the same degree {for $n \ge 2$}. Here, $\beta = (\beta_1,\beta_2,\beta_3)$ denotes a multi-index with  ${\vec x}^ \beta = x_1^{\beta_1}  x_2^{\beta_2}  x_3^{\beta_3} $, $\beta! = \beta_1! \beta_2! \beta_3 !$ and $\partial_{\vec x}^\beta (\cdot) = \partial_{x_1}^{\beta_1} \partial_{x_2}^{\beta_2} \partial_{x_3}^{\beta_3} (\cdot)$. The harmonic functions  $H_n^m({\vec x}) $ can be expressed as linear combinations of  ${\vec x}^\beta$ using
\begin{equation}
H_n^m({\vec x}) = \sum_{\beta, |\beta|=n} a_{\beta m  }^{\MH}  {\vec x}^\beta. \label{eqn:hinmon}
\end{equation}
Normalising $H_n^m({\vec x})$ such that the orthogonality property $\left < H_n^m({\vec x}), H_n^k({\vec x}) \right >_S= \delta_{m k}$ holds, where $\delta_{mk}$ is the Kronecker delta and $\left <u,v \right >_S=\int_S u \overline{v} \dif {\vec x}$ is the $L^2$ inner product over the surface of the unit sphere, and fixing $I_{n}^m ({\vec x})$ so that $\left < I_{n}^m ({\vec x}), I_{n}^k ({\vec x}) \right >_S= \delta_{mk}$, it follows that $\sum_{\ell=-n}^n a_{\ell m }^{\IH} \overline{a_{\ell k}^{\IH}} = \delta_{mk}$ and, hence,
the map from $ I_{\ell}^m ({\vec x})$  to  $H_n^m({\vec x}) $ is injective  with
\begin{equation}
I_n^\ell ({\vec x}) = \sum_{m = -n}^{n} \overline{a_{ \ell m }^{\IH} } H_{n}^m ({\vec x}). 
\end{equation}
On the other hand, the map from ${\vec x}^\beta$ to $H_n^m({\vec x})$ is not injective. Illustrative choices of $I_n^\ell ({\vec x}) $ for different degrees $n$ are presented in Table~\ref{tab:harmonicpolys}. The basis in this table does not satisfy $\left < I_{n}^m ({\vec x}), I_{n}^k ({\vec x}) \right >_S = \delta_{mk}$, but for the practical computations we will consider in Section~\ref{sect:sym}, this will not be required. An alternative basis satisfying $\left < I_{n}^m ({\vec x}), I_{n}^k ({\vec x}) \right >_S =  \delta_{mk}$ is provided in Table~\ref{tab:orthharmonicpolys} and a general approach for determining such a basis is presented by Karachik~\cite{karachik}.

\begin{table}
\begin{center}
\begin{tabular}{|c|c|l|}
\hline
$n$ & $2n+1$ & $I_n^\ell ({\vec x}) $ \\
\hline
0  & 1 & 1\\
\hline
1 & 3 & $x_1$, $x_2$, $x_3$\\
\hline
2 & 5 & $x_1^2 -x_2^2 $,  $x_1^2 -x_3^2,$\\
 {} & {} & $x_1 x_2$, $x_1 x_3,$\\
 {} & {} & $x_2 x_3$ \\
 \hline 
 3 & 7 & $x_1^3 -3 x_1 x_2^2  $, $ x_2^3 - 3x_1^2 x_2,$\\
 {} & {} & $ x_1^3  - 3 x_1 x_3^2 $, $  x_3^3 - 3 x_1^2 x_3, $\\
 {} & {} & $ x_2^3  - 3 x_2 x_3^2 $,  $  x_3^3 - 3  x_2^2 x_3 , $\\
{} & {} & $ x_1 x_2 x_3 $ \\
\hline
4 & 9 & $x_1^4 -6x_1^2 x_2^2  +x_2^4$, $x_1^4-6x_1^2 x_3^2 +x_3^4$,\\
{} & {} & $x_2^4-6x_2^2 x_3^2 +x_3^4$ , $x_1^3 x_2 - x_1 x_2^3$, \\
{} & {} &  $x_1^3 x_3 - x_1 x_3^3$,  $x_2^3 x_3 - x_2 x_3^3$, \\
{} & {} &  $3 x_1^2 x_2 x_3 - x_2 x_3^3$, $3 x_1 x_2^2 x_3 - x_1 x_3^3$, \\
{} & {} &   $3 x_1 x_2 x_3^2 - x_2 x_1^3$ \\
\hline
 \end{tabular}
\end{center}
\caption{Illustrative harmonic polynomials $I_n^\ell ({\vec x}) $  of different degrees $n$. Note that for this choice of basis $\left < I_{n}^m ({\vec x}), I_{n}^k ({\vec x}) \right >_S\ne \delta_{mk}$.} \label{tab:harmonicpolys}
\end{table}

\begin{table}
\begin{center}
\begin{tabular}{|c|c|l|}
\hline
$n$ & $2n+1$ & $I_n^\ell ({\vec x}) $ \\
\hline
0  & 1 & $\frac{1}{2\sqrt{\pi}}$ \\
\hline
1 & 3 & $\frac{1}{2}\sqrt{\frac{3}{\pi}}x_1$, $\frac{1}{2}\sqrt{\frac{3}{\pi}}x_2$, $\frac{1}{2}\sqrt{\frac{3}{\pi}}x_3$\\
\hline
2 & 5 & $\frac{1}{2} \sqrt{\frac{15}{\pi}} x_1 x_2$, $\frac{1}{2} \sqrt{\frac{15}{\pi}} x_2 x_3,$\\
 {} & {} &  $\frac{1}{2} \sqrt{\frac{15}{\pi}} x_1 x_3$, $\frac{1}{4} \sqrt{\frac{5}{\pi}} (x_1^2 -2 x_2^2 +x_3^2),$ \\
 {} & {} & $\frac{1}{4} \sqrt{\frac{15}{\pi}}  (x_1^2-x_3^2)$ \\
 \hline 
 3 & 7 & $ \frac{1}{4} \sqrt{\frac{35}{2 \pi}} (x_1^3 -3 x_1 x_2^2)$, $ \frac{1}{4} \sqrt{\frac{35}{2 \pi}} ( -3 x_1^2 x_2 +x_2^3),$ \\
 & & $  \frac{1}{4} \sqrt{\frac{21}{2 \pi}} x_1 (x_1^2 +x_2^2 -4 x_3^2 )$, $ \frac{1}{4} \sqrt{\frac{35}{2 \pi}} ( -3 x_1^2 x_3 +x_3^3 ),$\\
 & & $ \frac{1}{4} \sqrt{\frac{21}{2 \pi}} x_2 (x_1^2 +x_2^2 - 4  x_3^2)$, $ \frac{1}{4} \sqrt{\frac{21}{2 \pi}} x_3 (x_1^2 -4 x_2^2 + x_3^2),$\\
 & & $ \frac{1}{2} \sqrt{\frac{105}{ \pi}} x_1 x_2 x_3 $\\
\hline
4 & 9 & $\frac{3}{16} \sqrt{\frac{35}{\pi}} (x_1^4 -6 x_1^2 x_2^2 +x_2^4)$, $\frac{1}{16} \sqrt{\frac{5}{\pi}} ( 7 x_1^4 -x_2^4 +8x_3^4 +6x_1^2(x_2^2 -8x_3^2)),$\\
& &  $\frac{1}{4\sqrt{\pi}} ( -x_1^4 +4 x_2^4 -27 x_2^2 x_3^2 +4 x_3^4 +3x_1^2 ( x_2^2 +x_3^2))$, $ \frac{3}{4} \sqrt{\frac{35}{\pi}} x_1 x_2 (x_1^2 -x_2^2)$,\\
& & $ \frac{3}{4} \sqrt{\frac{35}{\pi}} x_1 x_3 (x_1^2 -x_3^2)$ , $ \frac{3}{4} \sqrt{\frac{35}{\pi}} x_2 x_3 (x_2^2 -x_3^2)$,\\
& & $ -\frac{3}{4} \sqrt{\frac{5}{\pi}} x_2 x_3 (-6x_1^2 + x_2^2 +x_3^2)$ , $ -\frac{3}{4} \sqrt{\frac{5}{\pi}} x_1 x_3 (x_1^2 -6x_2^2 +x_3^2)$ ,\\
& & $- \frac{3}{4} \sqrt{\frac{5}{\pi}} x_1 x_2 (x_1^2 +x_2^2-6x_3^2)$ \\
\hline
 \end{tabular}
\end{center}
\caption{Illustrative harmonic polynomials $I_n^\ell ({\vec x}) $  of different degrees $n$, which satisify $\left < I_{n}^m ({\vec x}), I_{n}^k ({\vec x}) \right >_S= \delta_{mk}$.} \label{tab:orthharmonicpolys}
\end{table}
\newpage
\section{Object Characterisation using GPTs, CGPTs and HGPTs} \label{sect:object}
 \subsection{Spherical and Taylor series expansions of $G({\vec x}, {\vec x}')$}
 By the addition formula for spherical harmonics (\ref{eqn:addsph}),  it can be shown that the Laplace free space Green's function $G({\vec x},{\vec x}') := 1/ ( 4 \pi | {\vec x} - {\vec x}'|)$ can be expressed as
 \begin{align}
 G({\vec x},{\vec x}')  = &  \sum_{n=0}^\infty  \frac{ |{\vec x}'|^n}{|{\vec x}|^{n+1}} \frac{1}{2n+1} \sum_{m=-n}^n Y_n^m(\theta,\phi) \overline{Y_n^m( \theta', \phi')}  \nonumber \\
 = &  \sum_{n=0}^\infty   \frac{1}{2n+1} \sum_{m=-n}^n K_n^m({\vec x}) \overline{H_n^m({\vec x}')} \label{eqn:gharmonic},
   \end{align}
for $|{\vec x }'|  < |{\vec x }| $. From the properties of $H_n^m({\vec x}')$ and $K_n^m({\vec x})$, we observe this expression is harmonic with respect to ${\vec x}'$ and ${\vec x}$, respectively. Furthermore, $ G({\vec x},{\vec x}') $ can be expressed in terms of  real valued harmonic polynomials as
\begin{align}
 G({\vec x},{\vec x}')  = & \sum_{n=0}^\infty   \frac{1}{2n+1} \frac{1}{{|\vec x }|^{2n+1}} \sum_{m=-n}^n  \sum_{\ell' = -n}^{n}  \sum_{\ell = -n}^{n}\overline{ a_{\ell ' m} ^{\IH}} I_{n }^{\ell'} ({\vec x}')  a_{\ell m} ^{\IH} I_{n }^{\ell} ({\vec x}) \nonumber \\
 = & \sum_{n=0}^\infty   \frac{1}{2n+1}  \frac{1}{{|\vec x }|^{2n+1}}    \sum_{\ell = -n}^{n}  I_{n }^{\ell} ({\vec x}') {I_{n}^{\ell} ({\vec x})} \nonumber ,
  \end{align}
since $ \sum_{m=-n}^n \overline{a_{\ell ' m} ^{\IH} }  a_{\ell m} ^{\IH}  =\delta_{\ell' \ell}$.

Alternatively, using (\ref{eqn:hinmon}) in (\ref{eqn:gharmonic}) gives
\begin{equation}
 G({\vec x},{\vec x}')  = \sum_{\beta, |\beta| =0}^ \infty  \frac{1}{2|\beta| +1} 
\sum_{m=-|\beta|}^{ | \beta|} K_{|\beta|}^m({\vec x})  \overline{a_{\beta m}^{\MH}} ({\vec x}')^\beta,
\end{equation}
  and, by comparing with the Taylor's series expansion  
    \begin{equation}
 G({\vec x},{\vec x}')  = \sum_{\beta, |\beta| =0}^ \infty  \frac{(-1)^{|\beta|}}{\beta !} \partial_{\vec x}^\beta G({\vec x},{\vec 0}) ( {\vec x}')^\beta \label{eqn:taylorg},
\end{equation}
   for $| {\vec x}'|$ in a compact set  and as  $|{\vec x}| \to \infty$, then
\begin{align}
 \frac{1}{2|\beta| +1} 
\sum_{m=-|\beta|}^{ | \beta|} K_{|\beta|}^m({\vec x}) \overline{ a_{\beta m}^{\MH} }= & \frac{(-1)^{|\beta|}}{\beta !} \partial_{\vec x}^\beta G({\vec x},{\vec 0})  \label{eqn:taylortoh}.
\end{align}

\subsection{Asymptotic expansion, GPTs and CGPTs}
For the problem stated in (\ref{eqn:model}), an  asymptotic expansion of  $(u - U)({\vec x})$ as $|{\vec x}| \to \infty$, has been derived by Ammari and Kang in their Definition 4.1~\cite{ammarikangbook}[pg 77]  and takes the form
\begin{equation}
(u - U)({\vec x})= \sum_{\alpha,\beta, |\alpha|=|\beta|=1}^\infty \frac{(-1)^{|\alpha|}} {\alpha! \beta !} \partial_{\vec x}^\alpha G({\vec x},{\vec 0}) M_{\alpha \beta } \partial^\beta U({\vec 0}) ,
\end{equation}
 for positions ${\vec x}$ away from  an inclusion $B$
where the generalised polarizability/polarisation tensor (GPT) coefficients that characterise $B$ are given by
\begin{equation}
M_{\alpha  \beta } : = \int_{\partial B} {\vec y}^\alpha \phi_\beta ({\vec y}) \dif {\vec y} , \qquad \phi_\beta ({\vec y}) := ( \lambda I - K_B^*)^{-1} ( {\vec \nu}_{\vec x} \cdot \nabla ({\vec x})^\beta ) ({\vec y}) ,\qquad  {\vec y} \in \partial B , \label{eqn:gpt}
\end{equation}
where $K_B^*$ denotes the $L^2$-adjoint of the Neumann-Poincar\'e operator $K_B$~\cite{ammarikangbook}[(2.20), pg. 18].

Consider the situation where the background field can be modelled {as} a point source located at the  position ${\vec x}^{\text{s}}$, so that  $U({\vec x}) = G({\vec x},{\vec x}^{\text{s}})$, and let $(u - U)({\vec x})$ be evaluated at position ${\vec x}^{\text{r}}$, far from the object. In this case, the measurements $V_{\text{sr}}$, corresponding to pairs of different sources and receivers, given by
\begin{equation}
V_{\text{sr}} =  \sum_{\alpha,\beta, |\alpha|=|\beta|=1}^\infty \frac{(-1)^{|\alpha| +|\beta|  }} {\alpha! \beta !} ( \partial_{\vec x}^\alpha G({\vec x},{\vec 0}))({\vec x}^{\text{r}}) M_{\alpha \beta } (\partial_{\vec x}^\beta  G({\vec x},{\vec 0}))({\vec x}^{\text{s}}),
\end{equation}
are of interest.
Then, using (\ref{eqn:taylortoh}),
\begin{align}
V_{\text{sr}} = &  \sum_{\alpha,\beta, |\alpha|=|\beta|=1}^\infty \frac{  1  } { (2|\alpha| +1)  (2|\beta| +1) } \sum_{m=-|\alpha|}^{ | \alpha|}  \sum_{n=-|\beta|}^{ | \beta|}
 K_{|\alpha|}^m({\vec x}^{\text{r}}) \overline{ a_{\alpha m}^{\MH}} M_{\alpha \beta} {a}_{\beta n}^{\MH}
   \overline{K_{|\beta|}^n({\vec x}^{\text{s}})  } \nonumber \\
    = & \sum_{p,q=1}^\infty   \sum_{m=-p }^{ p }  \sum_{n=-q}^{q }
 K_{p}^m({\vec x}^{\text{r}})  M_{qnpm}^{\text{C}}
  \overline{ K_{q}^n({\vec x}^{\text{s}})},     \label{eqn:exapandcgpt}
\end{align}
where, since right hand side of (\ref{eqn:taylortoh}) is real, the complex conjugate of both sides can be taken and is applied to $(\partial_{\vec x}^\beta  G({\vec x},{\vec 0}))({\vec x}^{\text{s}})$. In the above,
\begin{equation}
M_{qnpm}^{\text{C}} =   \sum_{\alpha, |\alpha|=p} \sum_{\beta, |\beta|=q}\frac{1}{ (2|\alpha| +1)  (2|\beta| +1)}  \overline{a_{\alpha m}^{\MH} }M_{\alpha \beta} {a}_{\beta n}^{\MH} ,
\end{equation}
are equivalent to the contracted GPT (CGPT) coefficients defined by Ammari,  Chung, Kang and Wang
\cite{ammari3dinvgpt}
and are expressed in terms of linear combinations of the GPT coefficients $M_{\alpha \beta}$.

Following Ammari {\it et al}~\cite{ammari3dinvgpt}, the matrices
\begin{equation}
({\mathbf M}_{pq})_{mn} := M_{qnpm}^{\text{C}}, \qquad -p \le m \le p, -q \le n \le q,
\end{equation}
are introduced, which are of dimension $(2p+1)\times (2q+1)$. We also introduce the $(2p+1) \times 1$ and $(2q+1)\times 1$ matrices ${\mathbf Y}_{\text{r}p}$ and ${\mathbf Y}_{\text{s}q}$ with entries
\begin{align}
({\mathbf Y}_{\text{r}p}  )_m:= & K_p^m ({\vec x}^{\text{r}}) ,\qquad -p \le m \le p,  \nonumber \\
({\mathbf Y}_{\text{s}q} ) _n:= & K_q^n ({\vec x}^{\text{s}}) ,  \qquad -q \le n \le q.  \nonumber 
\end{align}
Then, after truncating (\ref{eqn:exapandcgpt}) according to $p >N$ and $q > N$,
\begin{equation}
V_{\text{sr}} = \sum_{{p,q}=1}^N {\mathbf Y}_{\text{r}p}  {\mathbf M}_{pq} ({\mathbf Y}_{\text{s}q} )^*,
\end{equation}
where $*$ denotes the complex conjugate transpose $\overline{(\cdot)}^t$. Still further, Ammari {\em et al} introduce the block matrices ${\mathbf M}$ and ${\mathbf Y}$ with elements ${\mathbf M}_{ln}$  and $ {\mathbf Y}_{\text{r}p} $. In their Proposition 3.1 they show that ${\mathbf M}$ is hermitian and ${\mathbf M}_{nn}$ invertible for $n\ge 1$.

\subsection{Harmonic GPTs (HGPTs)}

A further alternative description of $V_{\text{sr}}$ is offered by using (\ref{eqn:himpoly}) so that
\begin{align}
V_{\text{sr}} = &  \sum_{p,q=1}^\infty  \frac{1}{{|{\vec x}^{\text{r}}|^{2p+1}} {|{\vec x}^{\text{s}}|^{2q+1}} } \sum_{m=-p }^{ p }  \sum_{n=-q}^{q }
 {H_{p}^m({\vec x}^{\text{r}})} M_{qnpm}^{\text{C}}
\overline{{H_{q}^n({\vec x}^{\text{s}})}}  \nonumber \\
= &  \sum_{p,q=1}^\infty   \frac{1}{{|{\vec x}^{\text{r}}|^{2p+1}} {|{\vec x}^{\text{s}}|^{2q+1}} }   \sum_{i=-p}^{ p }  \sum_{j=-q}^{q }
 {I_{p}^i({\vec x}^{\text{r}})}  M_{qjpi}^{\text{H}}
  {I_{q}^j ({\vec x}^{\text{s}})}  ,  \label{eqn:hproducts}
   \end{align}
where
\begin{subequations}
\begin{align}
 M_{qjpi}^{\text{H}} = &  \sum_{m=-p }^{ p }  \sum_{n=-q}^{q } a_{i m} ^{\IH} M_{qnpm}^{\text{C}} \overline{a_{jn}^{\IH}} \label{eqn:hgpt1} \\
= &   \sum_{m=-p }^{ p }  \sum_{n=-q}^{q }  \sum_{\alpha, |\alpha|=p} \sum_{\beta, |\beta|=q}\frac{1}{ (2|\alpha| +1)  (2|\beta| +1)}   a_{i m} ^{\IH} \overline{a_{\alpha m}^{\MH} }M_{\alpha \beta} {a}_{\beta n}^{\MH} \overline{a_{jn}^{\IH}}, \label{eqn:hgpt}
\end{align}
\end{subequations}
are the coefficients of what we call Harmonic GPTs (HGPTs). Note that since ${I_{q}^j ({\vec x}^{\text{s}})} = \overline{{I_{q}^j ({\vec x}^{\text{s}})}}$  the coefficients $M_{qjpi}^\text{H}$ are real for real $k$.
The HGPTs have $(2p+1) (2q+1)$ coefficients, the same  number as  the  CGPTs, but, as we will see, the HGPTs allow significant reductions in the number of independent coefficients for objects associated with a particular symmetry group.

In a similar way to ${\mathbf M}_{pq}$, the matrices ${\mathbf N}_{pq}$ with coefficients
\begin{equation}
({\mathbf N}_{pq})_{mn} := M_{qnpm}^{\text{H}}, \qquad -p \le m \le p, -q \le n \le q,
\end{equation}
are introduced, which are of dimension $(2p+1)\times (2q+1)$, and we call  HGPT matrices.
 We define the coefficients of  ${\mathbf I}_{\text{r} p}$ and ${\mathbf I}_{\text{s} q}$ as
\begin{align}
({\mathbf I}_{\text{r}p})_m:= & I_p^m ({\vec x}^{\text{r}}) , \qquad  -p \le m \le p, \nonumber \\
({\mathbf I}_{\text{s} q})_n:= & I_q^n ({\vec x}^{\text{s}})  ,  \qquad  -q \le n \le q,\nonumber 
\end{align}
so that, after truncating (\ref{eqn:hproducts}) corresponding to $p >N$ and $q > N$,
\begin{equation}
V_{\text{sr}} = \sum_{{p,q}=1}^N   \frac{1}{{|{\vec x}^{\text{r}}|^{2p+1}} {|{\vec x}^{\text{s}}|^{2q+1}} }  {\mathbf I}_{\text{r} p}  {\mathbf N}_{pq} ({\mathbf I}_{\text{s} q}  )^t.\label{eqn:Vsrharmonic}
\end{equation}
Block matrices ${\mathbf N}$ and ${\mathbf I}$ can also be introduced, with entries ${\mathbf N}_{pq}$ and ${\mathbf I}_{rp}$, in a similar way to ${\mathbf M}$ and ${\mathbf Y}$.
Then, in a similar manner to Proposition 3.1 in~\cite{ammari3dinvgpt}, we prove the following

\begin{proposition} \label{sect:propsymhgpt}
The HGPT matrix ${\mathbf N}$ satisfies   ${\mathbf N}= {\mathbf N}^*={\mathbf N}^t$. Furthermore, the matrices  ${\mathbf N}_{pp} $ are invertible for $p \ge 1$.
\end{proposition}
\begin{proof}
Given, $({\mathbf N}_{pq})_{ij} := M_{qjpi}^{\text{H}}$ then

\begin{align}
 M_{qjpi}^{\text{H}} = &  \sum_{m=-p }^{ p }  \sum_{n=-q}^{q } a_{i m} ^{\IH} M_{qnpm}^{\text{C}} \overline{a_{jn}^{\IH} }\nonumber \\
= &  \sum_{m=-p }^{ p }  \sum_{n=-q}^{q }  \sum_{\alpha, |\alpha|=p} \sum_{\beta, |\beta|=q}\frac{1}{ (2|\alpha| +1)  (2|\beta| +1)}   a_{i m} ^{\IH} \overline{a_{\alpha m}^{\MH}} M_{\alpha \beta} {a}_{\beta n}^{\MH} \overline{a_{jn}^{\IH}}  \nonumber \\
=&   \overline{
 \sum_{m=-p }^{ p }  \sum_{n=-q}^{q }  \sum_{\alpha, |\alpha|=p} \sum_{\beta, |\beta|=q}\frac{1}{ (2|\alpha| +1)  (2|\beta| +1)}   \overline{a_{i m} ^{\IH}} {{a}}_{\alpha m}^{\MH} M_{ \beta \alpha} \overline{a_{\beta n}^{\MH}} {a}_{jn}^{\IH}
}
=\overline{M_{piqj}^{\text{H}}}
\nonumber \\
={M_{piqj}^{\text{H}}},
\end{align}
which follows from using the symmetry property of $M_{\alpha \beta}$ on the coefficients of harmonic polynomials~\cite{ammarikangbook}, Thm. 4.10], and noting that the coefficients of $ M_{qjpi}^{\text{H}}$ are real for real $k$. Thus, ${\mathbf N}^* = {\mathbf N}^t$.

Again following~\cite{ammari3dinvgpt}, to show the invertibility of  ${\mathbf N}_{pp} $, it suffices to show that ${\mathbf v}^t {\mathbf N}_{pp} {\mathbf v} \ne 0 $ for any vector ${\mathbf v} \in {\mathbb R}^{2p+1}$, $ {\mathbf v} \ne 0$. Noting that we only need {to} consider the case of real ${\mathbf v}$ as   ${\mathbf N}_{pp} $ is real we get
\begin{align}
{\mathbf v}^t {\mathbf N}_{pp} {\mathbf v} = &  \sum_{i=-p}^p \sum_{j=-p}^p {v}_{i+p+1} M_{pjpi}^{\text{H}} v_{j+p+1} \nonumber \\
=& \sum_{i=-p}^p \sum_{j=-p}^p {v}_{i+p+1}  
   \sum_{m=-p }^{ p }  \sum_{n=-p}^{p }  \sum_{\alpha, |\alpha|=p} \sum_{\beta, |\beta|=p}\frac{1}{ (2|\alpha| +1)  (2|\beta| +1)}   a_{i m} ^{\IH} \overline{a_{\alpha m}^{\MH} } M_{\alpha \beta} {a}_{\beta n}^{\MH} \overline{a_{jn}^{\IH}} v_{j+p+1} \nonumber \\
 =&  \sum_{i=-p}^p \sum_{j=-p}^p {v}_{i+p+1}  
  \frac{1 }{(2p+1)^2}  \sum_{m=-p }^{ p }  \sum_{n=-p}^{p }  \sum_{\alpha, |\alpha|=p} \sum_{\beta, |\beta|=p}   a_{i m} ^{\IH} \overline{a_{\alpha m}^{\MH} } \int_{\partial B} {\vec y}^\alpha  \phi_\beta ({\vec y}) \dif {\vec y} {a}_{\beta n}^{\MH} \overline{a_{jn}^{\IH} } v_{j+p+1} \nonumber \\
  =&  \sum_{i=-p}^p \sum_{j=-p}^p {v}_{i+p+1}  
  \frac{ 1 }{(2p+1)^2}  \sum_{m=-p }^{ p }  \sum_{n=-p}^{p }     a_{i m} ^{\IH} \nonumber\\
  & \cdot  \int_{\partial B} \overline{H_p^m({\vec y})} ( \lambda I - K_B^*)^{-1} ( {\vec \nu}_{\vec x} \cdot \nabla ( H_p^n({\vec x}) ({\vec y}))  \dif {\vec y} \overline{a_{jn}^{\IH} }v_{j+p+1}\nonumber \\
  =&  \sum_{i=-p}^p \sum_{j=-p}^p {v}_{i+p+1}  
  \frac{1  }{ (2p  +1) ^2}       \int_{\partial B} I_p^i({\vec y}) ( \lambda I - K_B^*)^{-1} ( {\vec \nu}_{\vec x} \cdot \nabla ( I_p^j({\vec x}) ({\vec y})  )\dif {\vec y}  v_{j+p+1}\nonumber ,
  \end{align}
where $ {v}_{i+p+1}   I_p^i({\vec y})$ is a harmonic polynomial. Then, proceeding in a similar manner to Proposition 3.1 in~\cite{ammari3dinvgpt}, we find ${\mathbf v}^t {\mathbf N}_{pp} {\mathbf v} > 0 $ if $\lambda > 1/2$ and  ${\mathbf v}^t {\mathbf N}_{pp} {\mathbf v} < 0 $ if $\lambda \le 1/2$, completing the proof. 
\end{proof}

\begin{remark} \label{remark:indhgpt} {Proposition~\ref{sect:propsymhgpt} also implies that for the case of $p=q$ the number of independent coefficients of HGPTs and CGPTs reduce to $((2p+1)^2+(2p+1))/2 =(2p+1)(p+1)$, which is clearly less than $(2p+1)^2$.}
\end{remark}

\section{Transformation formulae for the HGPT matrix} \label{sect:hgptprop}

Following the results derived by~\cite{ammari3dinvgpt}, we present results for the scaling, shifting and rotation of the HGPT matrices. It is useful to introduce the $(p+1)\times (p+1) $ matrix with entries
\begin{equation}
({\mathbf A}_p^{\IH})_{mn} := a_{nm}^{\IH}, \qquad -p \le n \le p, -p \le m \le p , 
\end{equation}
which, by the results in Section~\ref{sect:spharmoics}, is unitary if $H_n^m({\vec x}) $ is chosen such that $ \left < H_n^m({\vec x}), H_n^k({\vec x}) \right >_S= \delta_{m k}$ and we fix  $I_{n}^m({\vec x})$ so that $\left < I_{n}^m ({\vec x}), I_{n}^k ({\vec x}) \right >_S= \delta_{mk}$, so that we can write 
\begin{align}
{\mathbf M}_{pq} = ({\mathbf A}_p^{\IH})({\mathbf N}_{pq} ) ({\mathbf A}_q^{\IH})^* , \label{eqn:mton}
\end{align}
and
\begin{align}
{\mathbf N}_{pq} = ({\mathbf A}_p^{\IH})^*({\mathbf M}_{pq} ) ({\mathbf A}_q^{\IH}) .\label{eqn:ntom}
\end{align}
If  a different choice of $H_n^m({\vec x}) $ is made then $({\mathbf A}_q^{\IH})^* $ should be replaced by $({\mathbf A}_q^{\IH})^{-1}$.

\subsection{Scaling}
\begin{lemma}
For any positive integers $\ell,n$ and the scaling parameter $s > 0$, the following holds:
\begin{equation}
{\mathbf N}_{\ell n } (sB) = s^{\ell + n+1} {\mathbf N}_{\ell n }(B) . \nonumber
\end{equation}
\end{lemma}
\begin{proof}
The result follows immediately from Lemma 4.1 in~\cite{ammari3dinvgpt} and (\ref{eqn:hgpt1}).
\end{proof}

\subsection{Shifting}
The shifting result in Lemma 4.2 of Ammari {\em et al.}~\cite{ammari3dinvgpt}, which we repeat below, concerns the shifting of CGPTs:
\begin{lemma}[Ammari {\em et al}~\cite{ammari3dinvgpt}]
For any positive integers $\ell,n$, and the shifting parameter ${\vec z}$, the following result holds:
\begin{equation}
{\mathbf M}_{\ell n }(B_z) = \sum_{i=1}^\ell \sum_{\nu=1}^n \overline{{\mathbf G}_{\ell i}({\vec z})} {\mathbf M}_{i \nu} {\mathbf G}_{n \nu}({\vec z})^t,
\end{equation}
where ${\mathbf G}_{n \nu } ({\vec z})$ is defined in (4.4) of Ammari {\em et al.}~\cite{ammari3dinvgpt}.
\end{lemma}
By using (\ref{eqn:mton}) and (\ref{eqn:ntom}), the above result can be applied to understand the shifting of HGPTs. 

\subsection{Rotation}
Defining a general rotation matrix ${\mathbf R}$ in terms of the Euler angles $\gamma,\beta, \alpha$ for rotations about the $x_1$, $x_2$ and $x_3$ axes, respectively, as
\begin{align}
{\mathbf R}= \left ( 
\begin{array}{ccc}
\cos \gamma & - \sin \gamma & 0 \\
\sin \gamma & \cos \gamma & 0 \\
0 & 0 & 1 \end{array} \right )
\left ( \begin{array}{ccc} 
\cos \beta & 0 & - \sin \beta \\
0 & 1 & 0 \\
\sin \beta & 0 & \cos \beta \\
\end{array} \right )
\left ( \begin{array}{ccc}
\cos \alpha & - \sin \alpha & 0 \\
\sin \alpha & \cos \alpha & 0 \\
0 & 0 & 1\end{array} 
\right ).
\end{align}

Ammari {\em et al.}~\cite{ammari3dinvgpt} in their  in Lemma 4.3, repeated below, describe how their CGPT transform under object rotation.

\begin{lemma}[Ammari {\em et al}~\cite{ammari3dinvgpt}] \label{lemma:rotten}
For a orthogonal matrix ${\mathbf R}$, the following relation holds
\begin{equation}
{\mathbf M}_{\ell n} (B_{\mathbf R}) = \overline{{\mathbf Q}_\ell ({\mathbf R})} {\mathbf M}_{\ell n} (B) {\mathbf Q}_n({\mathbf R})^t ,\nonumber
\end{equation}
where ${\mathbf Q}_n({\mathbf R})$ is called  a Wigner D-matrix.  ${\mathbf Q}_n({\mathbf R})$ is defined in (4.11) of~\cite{ammari3dinvgpt}.
\end{lemma}
By combining the above with expressions (\ref{eqn:mton}) and (\ref{eqn:ntom}) the transformation of  HGPTs under object rotation can be understood.

\section{Object characterisation and symmetry groups} \label{sect:sym}

As remarked in the introduction, the simplest form of a GPT is a rank 2 (P\'oyla-Szeg\"o) tensor, which characterises an object's shape $B$ and its contrast $k$ upto the best fitting ellipsoid. On the other hand, the complete set of GPTs uniquely defines both the shape and material contrast of the object~\cite{ammarikangbook}[pg. 90]. In practice, many of the physical objects that we wish to characterise have  rotational or reflectional symmetries. For example, if an object $B$ is invariant under the action of a rotation matrix ${\mathbf R}$, this means that ${\mathbf M}_{\ell n} (B_{\mathbf R}) ={\mathbf M}_{\ell n} (B)$ in Lemma~\ref{lemma:rotten}
 with a similar result for the matrix of HGPT coefficents ${\mathbf N}_{\ell n}$. If  an object has symmetries, many of the $(2 \ell +1)(2n +1)$ coefficients of
${\mathbf N}_{\ell n}$ will be zero and only a small number of independent coefficients will remain. In our earlier work, we have shown how the number of independent coefficients of a rank 2 polarizability tensor characterisation can be reduced if an object has rotational or reflectional symmetries~\cite{LedgerLionheart2015}.
This raises the question: What are the equivalent class of objects that an (H)GPT of a given order describes? To help to address this, we describe how knowing an object's symmetry group offers an effective way to deduce the independent coefficients of HGPT descriptions of different orders. Beforehand, we  recall how different symmetry groups can be distinguished and review the work of Meyer~\cite{meyer} on finding invariant harmonic polynomials under the action of a group.

\subsection{Types of finite groups}
Following Meyer~\cite{meyer}, we distinguish between different types of groups ${\mathfrak G}$. We consider those consisting of rotations only and those consisting of rotary inversions. A rotary inversion being a rotation followed by central symmetry with respect to a point on the axis of rotation. By choosing the fixed point to be the origin, the rotatory inversion is given by
 ${\mathbf J} {\mathbf R}$ where
\begin{equation}
{\mathbf J} = \left ( \begin{array}{ccc}
-1 & 0 & 0 \\
0  & -1 & 0 \\
0 & 0 & -1 \end{array}
\right ),  \nonumber
\end{equation}
and ${\mathbf R}$ is the rotation.

\subsubsection{Groups consisting of rotations only (Type 1)}
There are five classes of groups of this type and we indicate in the following how the rotational axes of each group are to be placed with respect to the $x_1$, $x_2$ and $x_3$ axes.
\begin{itemize}
\item ${\mathfrak C}_n$ Cyclic group, where the $n$-fold axis is taken to be the $x_3$-axis.
\item ${\mathfrak D}_n$ Dihedral group, where the $n$-fold axis is taken to be the $x_3$-axis and one of the 2-fold axes is taken to be the $x_1$-axis.
\item ${\mathfrak T}$ Tetrahedral group, which consists of rotations that transform a regular tetrahedron to itself. We follow Meyer  who  assumes the tetrahedron is placed so that its 3-fold axes coincide with the 3-fold axes of ${\mathfrak O}$ (below), and the 2-fold axes are taken as the coordinate axes.
\item ${\mathfrak O}$ Octahedral group, which consists of the rotations that transform a cube (or a regular octahedron) onto itself.  We follow Meyer who assumes the cube is placed with its centre at the origin and with its faces parallel to the coordinate axes.
\item ${\mathfrak I}$ Icosahedral group, which consist of the rotations that transform a regular icosahedron (or regular dodecahedron) into itself. Again, we follow Meyer who assumes the icosahedron is placed with its centre at the origin, and the coordinate axes pass through the midpoint of opposite edges such that the edges through which the $x_1$ axis passes are parallel to the $x_2$ axis.

\end{itemize}

\subsection{Groups containing rotary-inversions}
There are two groups containing rotary inversions:
\begin{itemize}
\item Type 2 are those with centre of symmetry, which are obtained by adjoing $\mathbf{J}$ to a group of type 1. The order of these groups is twice that of the corresponding rotational group. The groups of this type are denoted by ${\mathfrak C}_{n,i}$, ${\mathfrak D}_{n,i}$, ${\mathfrak T}_i$, ${\mathfrak O}_i$ and ${\mathfrak I}_i$.

\item Type 3 are derived from a rotational group ${\mathfrak G}_2$, which has a subgroup, ${\mathfrak G}_1$, of index 2.

\end{itemize}

\subsection{Harmonic polynomials invariant under the action of a group}
Meyer describes an approach for determining all the harmonic polynomials $I_m^{ i , {\mathfrak G}} ({\vec x})$ of degree $m$ that are invariant under the action of a given symmetry group ${\mathfrak G}$ where we expect far few than $2m+1$ invariant polynomials.  
He explains that the number of elements $g_m$ of an invariant basis 
of degree $m$ for a group ${\mathfrak G}$  of order $n$ of orthogonal matrices can be obtained from the generating function using a result of Molien~\cite{molien} as
\begin{equation}
g(t) =\sum_{m=0}^\infty g_m t^m .
\end{equation} 
 However, the interest lies in an invariant harmonic basis. Meyer's main theorem addresses this question:
\begin{theorem}[Meyer~\cite{meyer}]\label{thm:meyer1}
Let ${\mathfrak G} $ be a finite group of orthogonal linear transformations in $x_1$, $x_2$ and $x_3$. Let
\begin{equation}
h(t) = \sum_{m=0}^\infty h_m t^m ,
\end{equation}
in which $h_m$ is the number of elements in an invariant harmonic basis for ${\mathfrak G}$ of degree $m$. Then 
\begin{equation}
h(t) = (1-t^2) g(t),
\end{equation}
where $g(t)$ is the generating function of Molien.
\end{theorem}
\begin{remark}
 Some of these sequences $h_m$ are listed in the online encyclopedia of integer sequences~\cite{intsequences}.
\end{remark}

The invariant harmonic polynomials  can be obtained by applying another result of Meyer:
\begin{theorem}[Meyer~\cite{meyer}]\label{thm:meyerpoly}
Let 
\begin{equation}
Q_1(x_1,x_2,x_3), \ldots,Q_{h_m}(x_1,x_2,x_3), \nonumber
\end{equation}
be $h_m$ homogeneous harmonic (operating) polynomials of degree $m$, which are invariants of ${\mathfrak G}$ and which are linearly independent $\text{mod }r^2=\text{mod }(x_1^2 +x_2^2 +x_3^2)$ then the $h_m$ harmonic functions 
\begin{equation}
r^{2m+1} Q_j \left ( \frac{\partial}{\partial x_1},  \frac{\partial}{\partial x_2}, \frac{\partial}{\partial x_3} \right ) \frac{1}{r}, \qquad j=1,\ldots,h_m,
\end{equation}
form an invariant harmonic basis of degree $m$ for ${\mathfrak G}$.
\end{theorem}

The harmonic polynomials $I_m^{i,{\mathfrak G}}({\vec x})$, $i=1,\ldots, h_m$ of degree $m$ that are invariant under the action of the group ${\mathfrak G}$ produced by the above will be linear combinations of $I_m^\ell({\vec x})$ so that
\begin{equation}
I_m^{i, \mathfrak G}({\vec x}) = \sum_{\ell =-m}^m a_{\ell i}^{I, {\mathfrak G}} I_m^\ell ({\vec x}), \qquad i =1, \ldots, h_m  \label{eqn:ItoS}
.
\end{equation}

We illustrate Meyer's approach by considering the object shown in Figure~\ref{fig:block}.
\begin{figure}
\begin{center}
\includegraphics[width=1in]{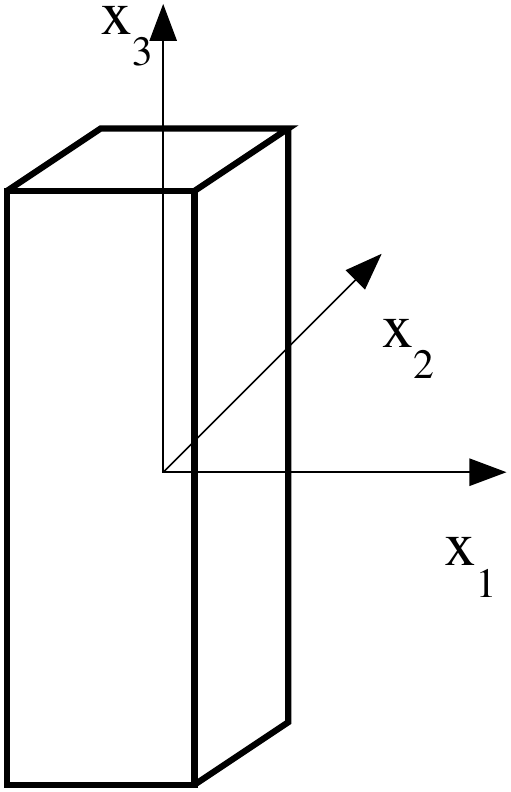}
\end{center}
\caption{Illustration of an object $B$, which has the symmetry group ${\mathfrak D}_4$.} \label{fig:block}
\end{figure}
This object has a 4-fold rotational symmetry about $x_3$ and 2-fold rotational symmetries about the $x_1$ and $x_2$ axes. This means the appropriate symmetry group for the object is ${\mathfrak G} = {\mathfrak D}_4$.  In his Table I, Meyer presents $h(t)$ for different groups, with ${\mathfrak D}_4$ being of the form
\begin{equation}
h(t) =(1-t^2)g(t) =\frac{1}{(1-t^2)} \frac{1+t^5}{1-t^4} = 1 + t^2 +2t^4 +t^5+ \ldots,
\end{equation}
which tells us $h_m=1$ for $m=0,2,5$, so there is just a single invariant harmonic polynomial, $I_m^{1,{\mathfrak D}_4} ({\vec x})$ for these degrees, while $h_m=2$ for $m=2$, and so we have   $I_m^{j ,{\mathfrak D}_4} ({\vec x})$, $j=1,2$,
 for this degree, and $h_m=0$ for $m=1,3 $, and so we have no invariant harmonic polynomials for these degrees.

To construct the $I_m^{j,{\mathfrak G}} ({\vec x})$ for each degree $m$, we refer to Table III in Meyer's article where he describes how to construct the operating polynomial for each group ${\mathfrak G}$. In our case,  the operating polynomials $Q_1(x_1,x_2,x_3), \ldots, Q_{h_m}(x_1,x_2,x_3)$ for ${\mathfrak G}={\mathfrak D}_4$ are constructed from $x_3^2$, $C_{un}$,  $x_3 C_{un}'$, $C_{vn}$, $x_3 C_{un}'$, $u=1,3,5,\ldots$, $v=2,4,6,\ldots$, where
\begin{align}
C_n =& x_1^n - \left ( \begin{array}{c} n \\ 2 \end{array} \right ) x_1^{n-2} x_2^2 + \left ( \begin{array}{c} n \\ 4 \end{array} \right ) x_1^{n-4} x_2^4- \ldots, \nonumber\\
C_n' =&\left ( \begin{array}{c} n \\ 1 \end{array} \right ) x_1^{n-1} x_2 - \left ( \begin{array}{c} n \\ 3 \end{array} \right ) x_1^{n-3} x_2^3 + \left ( \begin{array}{c} n \\ 5 \end{array} \right ) x_1^{n-5} x_2^5 - \ldots ,\nonumber
\end{align}
In Table~\ref{tab:invd4} we summarise the results of the calculation for different $m$ and the group ${\mathfrak D}_4$. Observe that the results $ I_m^{j,{\mathfrak D}_4} ({\vec x})$, $j=1,\ldots,h_m,$ are linear combinations of the $I_m^\ell({\vec x})$, $\ell = -m,\ldots,m$ of the same degree $m$ presented in Table~\ref{tab:harmonicpolys}, as expected.

\begin{table}
\begin{center}
\begin{tabular}{|c|c|c|c|}
\hline
$m$ & $h_m$ &  $Q_j(x_1,x_2,x_3)$, $j=1,\ldots,h_m$ & $I_m^{j,{\mathfrak D}_4} ({\vec x}) =r^{2m+1} Q_j \left ( \frac{\partial}{\partial x_1},  \frac{\partial}{\partial x_2}, \frac{\partial}{\partial x_3} \right ) \frac{1}{r} $ \\
\hline
0 & 1 & $1$ &$1$ \\
\hline
1 & 0 & - & - \\
\hline
2 & 1 & $x_3^2$ & $2x_3^2 -x_1^2 -x_2^2$ \\
\hline
3 & 0 & - & - \\
\hline
4 & 2 & $x_3^4$ &$ 24x_3^4 +9( x_1^4 + x_2^4) -72 ( x_1^2x_3^2 + x_2^2 x_3^2 ) +18x_1^2 x_2^2$\\ 
{} & {} & $C_4 = x_1^4 -6x_1^2 x_2^2$  & $105( x_1^4 +x_2^4)  -630 x_1^2 x_2 ^2 $ \\
\hline
\end{tabular}
\end{center}
\caption{Illustration of harmonic polynomials $I_m^{j, {\mathfrak G}} ( {\vec x}) $, $j=1,\ldots,h_m$, of degrees $m=0,1,2,3,4$, which are invariant under the symmetry group ${\mathfrak G}= {\mathfrak D}_4$.} \label{tab:invd4}
\end{table}

\subsection{Symmetric products of harmonic polynomials invariant under the action of a symmetry group}

From (\ref{eqn:gpt}) and (\ref{eqn:hgpt1}) we see that the coefficients of the HGPTs have the form
\begin{equation}
M_{qjpi}^{\text{H}} =C(p,q) \int_{\partial B} I_p^i({\vec y})   
 ( \lambda I - K_B^*)^{-1} ( {\vec \nu}_{\vec x} \cdot \nabla_{\vec x} I_q^j ({\vec x}) ) ({\vec y}) \dif {\vec y} \nonumber ,
\end{equation} 
where $C(p,q)$ depends on $p$ and $q$. Now consider the HGPT characterisation of $B$ under the action of a rotation matrix ${\mathbf R}$ as
\begin{align}
M_{qjpi}^{\text{H}} ({\mathbf R} (B))=&C(p,q) \int_{\partial {\mathbf R} (B)} I_p^i( {\vec y})   
 ( \lambda I - K_B^*)^{-1} ( {\vec \nu}_{\vec x} \cdot \nabla_{\vec x} I_q^j ({\vec x}) ) ({\vec y}) \dif {\vec y} \nonumber \\
=& C(p,q) \int_{\partial B} I_p^i({\mathbf R}{\vec y} )  
 ( \lambda I - K_B^*)^{-1} ( {\vec \nu}_{\vec x} \cdot \nabla_{\vec x} I_q^j ({\vec x}) ) ({\mathbf R} {\vec y}) \dif {\vec y} \nonumber \\
 =& C(p,q) \int_{\partial B} I_p^i({\mathbf R} {\vec y})  
 ( \lambda I - K_B^*)^{-1} ( {\vec \nu}_{\vec x} \cdot \nabla_{\vec x} I_q^j ({\mathbf R} {\vec x}) ) ({\vec y}) \dif {\vec y} \nonumber ,
\end{align}
which follows since $ ( \lambda I - K_B^*)$ is invariant under the rotation map~\cite{ammari2012}{Prop. 4.1].  

Furthermore, since the HGPT coefficients $M_{qjpi}^{\text{H}}$ appear together with products of harmonic polynomials, as expressed in (\ref{eqn:hproducts}),   an  object rotation is equivalent to a rotation of the coordinate system for both the excitation and measurement so that
\begin{equation}
\sum_{i=-p}^{ p }  \sum_{j=-q}^{q }
 {I_{p}^i({\mathbf R}{\vec x})}  M_{qjpi}^{\text{H}}  
  {I_{q}^j ({\mathbf R}{\vec y})}
  \equiv
\sum_{i=-p}^{ p }  \sum_{j=-q}^{q }
 {I_{p}^i({\vec x})}  M_{qjpi}^{\text{H}} ({\mathbf R} (B)) 
  {I_{q}^j ({\vec y})}.
  \end{equation}
In addition, given the symmetry property of HGPTs described in Proposition~\ref{sect:propsymhgpt}, determining those HGPT coefficients $M_{qjpi}^H$ that are invariant under the action of a symmetry group reduces to finding symmetric products of harmonic polynomials  $I({\vec x}), J({\vec x})$  (of  possibly different degrees $p$ and $q$)  in the form 
\begin{align}
S({\vec x},{\vec y}) = S({\vec y},{\vec x})=I({\vec x}) J({\vec y}) + I({\vec y})J({\vec x}) ,  \nonumber
\end{align}
 that have the property that
\begin{align}
S({\mathbf R}{\vec x},{\mathbf R}{\vec y}) = S({\vec x},{\vec y}) , \label{eqn:fixs}
\end{align}
for all matrix representations ${\mathbf R}$ that make up the group ${\mathfrak G}$. Applying the terminology of representation theory (see Appendix~\ref{sect:app} for a brief summary of the key results we require), we say (\ref{eqn:fixs}) means that $S({\vec x},{\vec y})$ is fixed by ${\mathfrak G}$.

Meyer's results in Theorems~\ref{thm:meyer1} and~\ref{thm:meyerpoly} determine the number of invariant harmonic polynomials  that are fixed by ${\mathfrak G}$ and he provides a methodology to determine a basis for these. However, for HGPTs, we need to find the number of  symmetric harmonic polynomials $S({\vec x},{\vec y})$  that are fixed by ${\mathfrak G}$ and, also, to establish  a methodology to determine a basis for  these invariant products. Therefore, we follow a different approach to Meyer.

We introduce 
\begin{equation}
{\vec S}_{pq} :=\{ I_p^i({\vec x}) I_q^j({\vec y}) + I_p^i({\vec y})I_q^j({\vec x}) , -p\le i\le p ,-q\le j\le q, {\vec x}\in {\mathbb R}^3, {\vec y}\in {\mathbb R}^3 \}, \nonumber
\end{equation}
for the vector space provided by the symmetric products of harmonic polynomials of degree $p$ and $q$ and denote by
\begin{equation}
{\vec S}_{pq}^{\mathfrak G} :=\{ {\vec S} \in {\vec S}_{pq}   : {\vec \pi} ( {\mathbf R})({\vec S}) ={\vec S}  \ \forall {\mathbf R} \in {\mathfrak G}  \} \subseteq {\vec S}_{pq} ,  \nonumber
\end{equation}
the subspace corresponding to those elements of ${\vec S}_{pq}$, which are fixed ${\mathfrak G}$, with ${\mathfrak G}$ having the representation $({\vec \pi}, {\vec S}_{pq})$ where ${\vec \pi}$ is a homomorphism from ${\mathfrak G}$ to the group of invertible linear transformations of $ {\vec S}_{pq}$.

Based on the results in  Appendix~\ref{sect:app}, we provide Algorithm~\ref{alg:getsymprodbasis} for determining the dimension of ${\vec S}_{pq}^{\mathfrak G}$ and a basis for $ {\vec S}_{pq}^{\mathfrak G}$~\footnote{We have implemented  Algorithm~\ref{alg:getsymprodbasis} in Mathematica.  Our implementation
 available at \texttt{https://github.com/pdledger/HGPTSymmetries} (will be made public on publication)
 follows these steps, but it is considerably more involved  (given in part our limited knowledge of Mathematica's functionality)}.

\begin{algorithm}
\caption{Algorithm for determining the dimension and a basis for $ {\vec S}_{pq}^{\mathfrak G}$. } \label{alg:getsymprodbasis}
\begin{algorithmic}[1]
\Require${\vec S}_{pq}  $ containing symmetric products of harmonic polynomials of degrees $p$ and $q$, from Table~\ref{tab:harmonicpolys}, and rotation matrices ${\mathbf R}_i$, $ i=1,\ldots,n$, describing the
 group ${\mathfrak G}$ of order $n$.
\Ensure The dimension of  $ {\vec S}_{pq}^{\mathfrak G}$ and a basis for $ {\vec S}_{pq}^{\mathfrak G}$.
\State Determine ${\mathbf M}_\pi=\frac{1}{n} \sum_{i=1}^n {\vec \pi} ({\mathbf R}_i)$.
\State The dimension of  ${\vec S}_{pq}^{\mathfrak G}$ is $m=\Tr ({\mathbf M}_\pi) $.
\State A basis for ${\vec S}_{pq}^{\mathfrak G}$ is the first $m$ independent elements of $ {\mathbf M}_\pi {\vec S}_{pq}$.
\end{algorithmic}
\end{algorithm}

In the case of a dihedral group ${\mathfrak D}_n$ of order $n$, one could apply Algorithm~\ref{alg:getsymprodbasis} first for ${\mathfrak C}_n$ (for rotations about $x_3$) to obtain ${\vec S}_{pq}^{{\mathfrak C}_n}$ and then again for ${\mathfrak C}_2$ (for rotations about $x_1$)  to generate ${\vec S}_{pq}^{{\mathfrak C}_2^{x_1}}$.  Determining a basis for   ${\mathfrak D}_n$ then reduces to  determining a basis for the intersection  ${\vec S}_{pq}^{{\mathfrak C}_n} \cap {\vec S}_{pq}^{{\mathfrak C}_2^{x_1}}$. Alternatives for determining the intersection of these subspaces  include Algorithm 12.4.3 provided by Golub and van Loan in~\cite{golub}[pg. 604], part of the Zassenhaus algorithm~\cite{zassenhaus} or alternatively  Algorithm~\ref{alg:subspace} could be applied to determine the intersection between two subspaces of a vector space.

\begin{algorithm}
\caption{Algorithm for determining the intersection of two-subspaces of a linear vector space.} \label{alg:subspace}
\begin{algorithmic}[1]
\Require Two subspaces with given basis for ${\mathbb R}^n$ provided in ${\mathbf A} \in {\mathbb R}^{n \times p}$ and $ {\mathbf B}\in {\mathbb R}^{n \times q}$ with $p,q \le n$.
\Ensure A basis for intersection of the subspaces.
\State Create ${\mathbf C} = \left ( \begin{array}{cc} {\mathbf A} & {\mathbf B} \end{array} \right ) \in {\mathbb R}^{n \times (p+q)}$.
\State Determine the null space ${\mathbf N}$ of ${\mathbf C}$.
\State Then the rows of either ${\mathbf N}(:,1:p) {\mathbf A}^T$ or ${\mathbf N}(:,p+1:p+q) {\mathbf B}^T$ provide  a basis for the intersection, where $:$ indicates the complete set of row (or column) elements and $1:p$ the elements between 1 and $p$.
\end{algorithmic}
\end{algorithm}

\begin{remark}
Once ${\vec S}_{pq}^{\mathfrak G}$ is identified, the invariant HGPT coefficients are immediate.
\end{remark}

\subsection{Application of Algorithm~\ref{alg:getsymprodbasis} for ${\vec S}_{11}$ and ${\mathfrak C}_4$}\label{sect:example}

We consider an illustration of Algorithm~\ref{alg:getsymprodbasis} for ${\vec S}_{11}$, which, using Table~\ref{tab:harmonicpolys}, explicitly has the form
\begin{equation}
{\vec S}_{11}= \{ x_1 y_1, x_1y_2 +x_2y_1, x_1 y_3 +x_3 y_1, x_2 y_2, x_2 y_3 +x_3 y_2 ,  x_3 y_3 \} , \nonumber 
\end{equation}
and
 the ${\mathfrak C}_4$ group, which is described by the rotation matrices
\begin{equation}
{\mathbf R}_1 =\left ( \begin{array}{rrr} 1 & 0 & 0 \\ 0 & 1 & 1 \\ 0 &  0 & 1 \end{array} \right ), \ {\mathbf R}_2 =\left ( \begin{array}{rrr}  0& -1 & 0 \\ 1 & 0 & 0 \\ 0 &  0 & 1 \end{array} \right ) , \
{\mathbf R}_3 =\left ( \begin{array}{rrr}  -1& 0 & 0 \\ 0 & -1 & 0 \\ 0 &  0 & 1 \end{array} \right ), \ {\mathbf R}_4 =\left ( \begin{array}{rrr}  0& 1  & 0 \\ -1 & 0 & 0 \\ 0 &  0 & 1 \end{array} \right ) 
. \nonumber
\end{equation}
Then, by following the ordering in ${\vec S}_{11}$,  ${\vec \pi } ({\mathbf R}_1)={\mathbf I}_{6\times 6}$ is just the $6 \times 6$ identity matrix and 
\begin{equation}
 {\vec \pi } ({\mathbf R}_2)  = \left ( \begin{array}{rrrrrr}
0 & 0 & 0 & 1 & 0 & 0\\
0 & -1& 0 & 0 & 0 & 0\\
0 & 0 & 0 & 0& -1 & 0\\
1 & 0 & 0 & 0 & 0 & 0\\
0 & 0 & 1 & 0 & 0 & 0 \\
0 & 0 & 0 & 0 & 0 & 1 \end{array} \right ), \
 {\vec \pi } ({\mathbf R}_3)  = \left ( \begin{array}{rrrrrr}
1 & 0 & 0 & 0 & 0 & 0\\
0 & 1& 0 & 0 & 0 & 0\\
0 & 0 & -1 & 0& 0 & 0\\
0 & 0 & 0 & 1 & 0 & 0\\
0 & 0 & 0 & 0 & -1 & 0 \\
0 & 0 & 0 & 0 & 0 & 1 \end{array} \right ) , \nonumber
\end{equation}
\begin{equation}
 {\vec \pi } ({\mathbf R}_4)  = \left ( \begin{array}{rrrrrr}
0 & 0 & 0 & 1 & 0 & 0\\
0 & -1& 0 & 0 & 0 & 0\\
0 & 0 & 0 & 0& 1 & 0\\
1 & 0 & 0 & 0 & 0 & 0\\
0 & 0 & -1 & 0 & 0 & 0 \\
0 & 0 & 0 & 0 & 0 & 1 \end{array} \right ), \nonumber
\end{equation}
are obtained by applying ${\mathbf R}_2, {\mathbf R}_3$ and ${\mathbf R}_4$ to each element of ${\vec S}_{11}$. Their average is
\begin{equation}
{\mathbf M}_\pi = 
\frac{1}{4}  \left ( \begin{array}{rrrrrr}
2 & 0 & 0 & 2 & 0 & 0\\
0 & 0 & 0 & 0 & 0 & 0\\
0 & 0 & 0 & 0& 0 & 0\\
2 & 0 & 0 & 2 & 0 & 0\\
0 & 0 & 0 & 0 & 0 & 0 \\
0 & 0 & 0 & 0 & 0 & 4 \end{array} \right ), \nonumber
\end{equation}
and, hence, it follows that the dimension of ${\vec S}_{11}^{{\mathfrak C}_4} $ is $m=\Tr ( {\mathbf M}_{\pi } )=2$ and a basis for ${\vec S}_{11}^{{\mathfrak C}_4} $ are  the 2  independent elements of
\begin{equation}
{\mathbf M}_\pi {\vec S}_{pq} = \frac{1}{2} \left ( \begin{array}{c} 
 x_1 y_1+ x_2 y_2 \\ 0 \\ 0 \\ x_1 y_1 +  x_2 y_2 \\ 0\\ 2 x_3 y_3 \end{array} \right ).
 \end{equation}
Thus, by removing unnecessary constants, ${\vec S}_{11}^{{\mathfrak C}_4}= \{ x_1 y_1+  x_2 y_2,   x_3 y_3 \} $. Note that Table~\ref{tab:orthharmonicpolys}  could be equivalently used to define ${\vec S}_{11}$, and this will result in different HGPT coefficients, but will result {in} an equivalent reduction in dimensions.

\begin{remark} \label{tab:grptab}
In a similar way to the example above, the sets of ${\vec S}_{pq}^{\mathfrak G}$ for the groups ${\mathfrak G}= {\mathfrak C}_n$, $n=2,3,4,5,6$, together with their dimensions, can be obtained from Algorithm~\ref{alg:getsymprodbasis} and the results for different $p,q$ are presented in Tables~\ref{tab:symc2}, ~\ref{tab:symc3}, ~\ref{tab:symc4}, ~\ref{tab:symc5} and ~\ref{tab:symc6}.
{The corresponding sets of ${\vec S}_{pq}^{\mathfrak G}$ for the groups ${\mathfrak G}= {\mathfrak D}_n$, $n=2,3,4,5,6$,  ${\mathfrak G}= {\mathfrak I}$, ${\mathfrak T}$  and ${\mathfrak O}$ can be obtained using our software  }\texttt{https://github.com/pdledger/HGPTSymmetries}.
In the case of the cyclic and dihedral groups, and the considered polynomial degrees, there is no change in  ${\vec S}_{pq}^{\mathfrak G}$ for $n \ge 6$. This indicates that  to distinguish a particular order cyclic (or dihedral) group and its lower counterparts we need to {have} HGPTs of sufficient degree.  
 \end{remark}

\subsection{Examples of  $ {\vec S}_{pq}^{\mathfrak G}$ for different groups ${\mathfrak G}$}

Considering magnetostatic measurements using magnetometry (or metal detection measurements at very low frequencies, where only the magnetic part of the characterisation can be recorded), one potential application is to discriminate between unexploded ordnance (UXO) and shrapnel buried in the ground in regions of former conflict. While shrapnel will be random in shape and, hence, lack geometrical symmetries, common forms of UXO are likely to have cyclic symmetries. To illustrate this, some images of decommissioned  mortar bomb shell casings and hand (fragmentation) grenades, which were taken during a visit to the \href{https://www.visitcambodiatravel.com/cambodia-war-remnant-museum.html}{Cambodia War Remnant Museum} by one of the authors,  are shown in Figure~\ref{fig:uxo}. In this figure, the exterior shape of the objects shown can be characterised by the cyclic group of different orders including ${\mathfrak C}_4$, ${\mathfrak C}_6$, ${\mathfrak C}_8$, ${\mathfrak C}_{10}$, but there may be internal structures or damage that breaks the symmetry. {The sketches of a BLU-61 submunition shown in Figure~\ref{fig:blu61} illustrate the object's mirror symmetry as well as its 4--fold rotational symmetry and, hence, when suitably orientated, the object's exterior shape is characterised by the ${\mathfrak D}_4$ group. While this figure illustrates that the dihedral group may also be useful for describing some UXO components,}  the tetrahedral, octahedral and icosahedral symmetry groups are less likely to be so, but may still be useful for other forms of object detection (e.g. in crystallography or   in other forms of detection that the authors are yet to consider). Further applications include in the characterisation of conducting objects in EIT, ERT and by electrosensing fish, as described in the introduction.
\begin{figure}
\begin{center}
$\begin{array}{cc}
\includegraphics[width=2in]{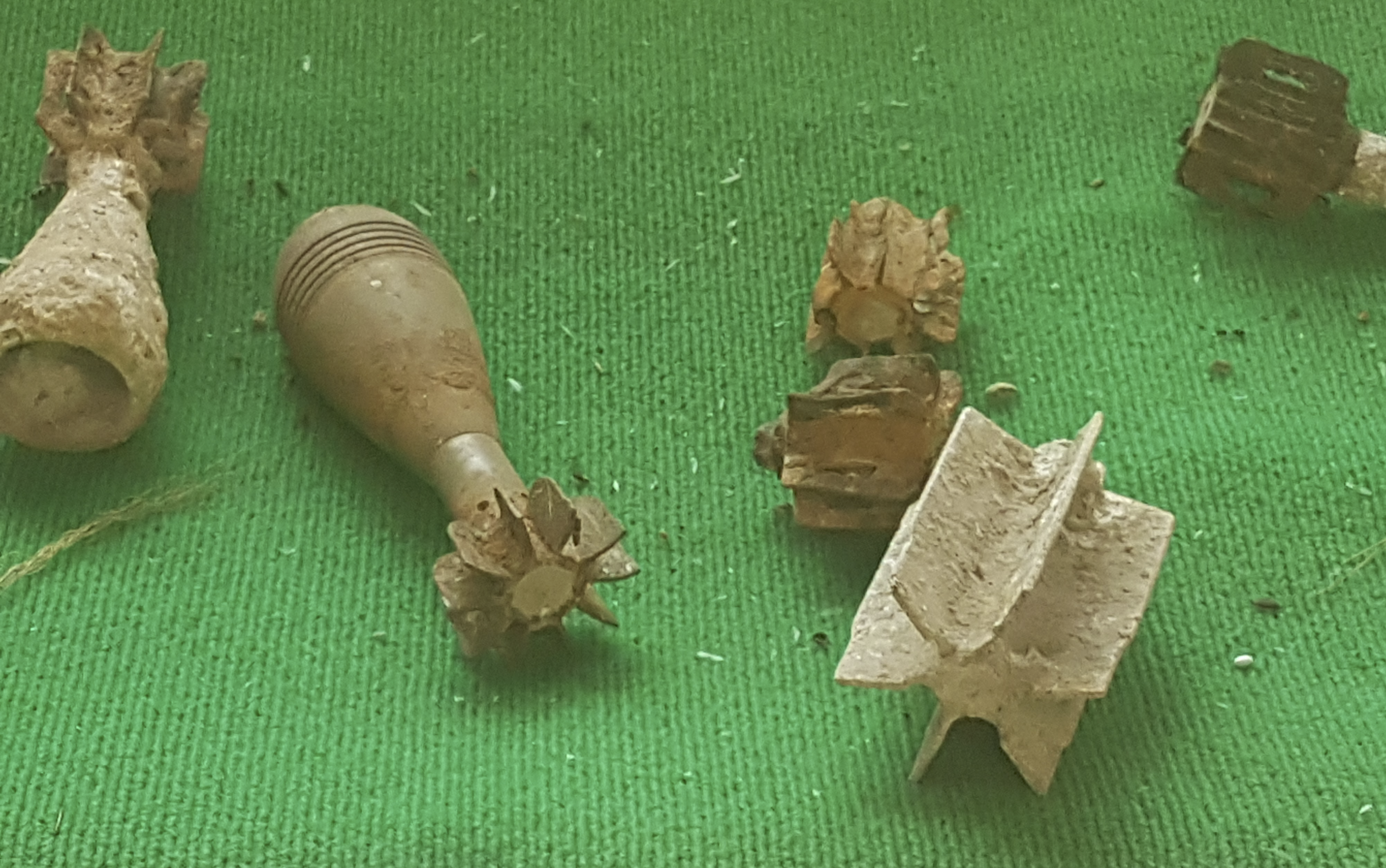} &
\includegraphics[width=2in]{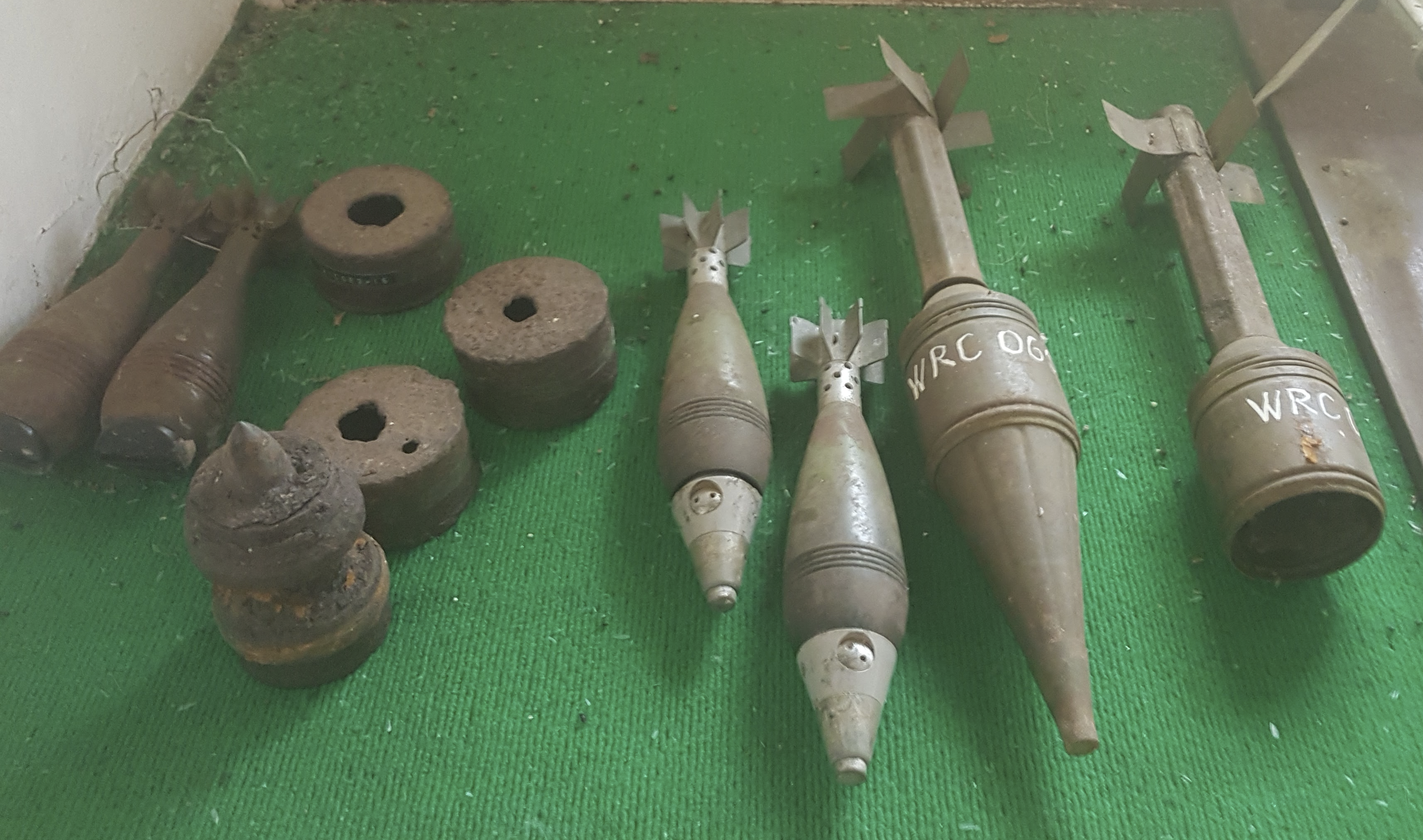} \\
\includegraphics[width=2in]{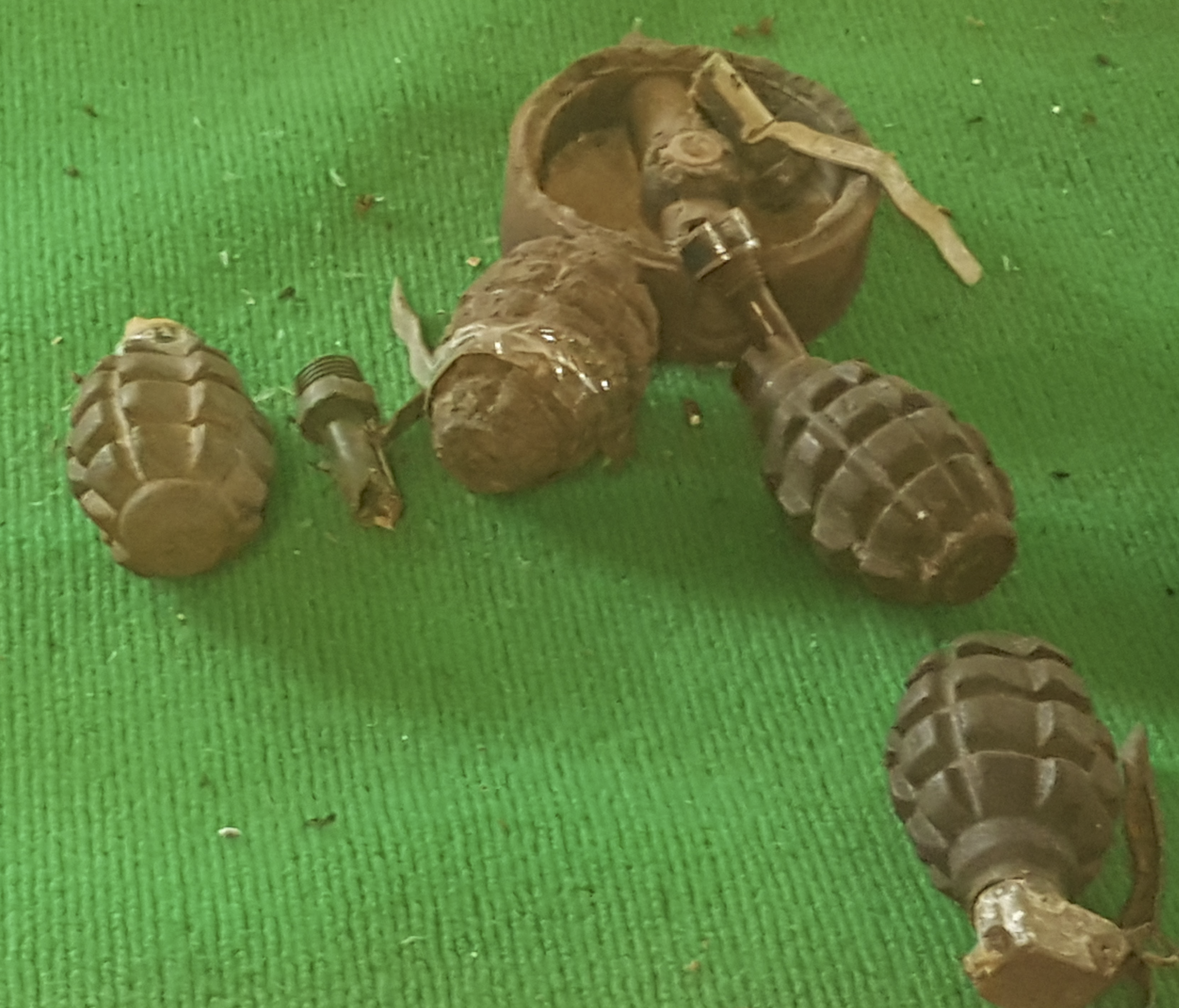} &
\includegraphics[width=2in]{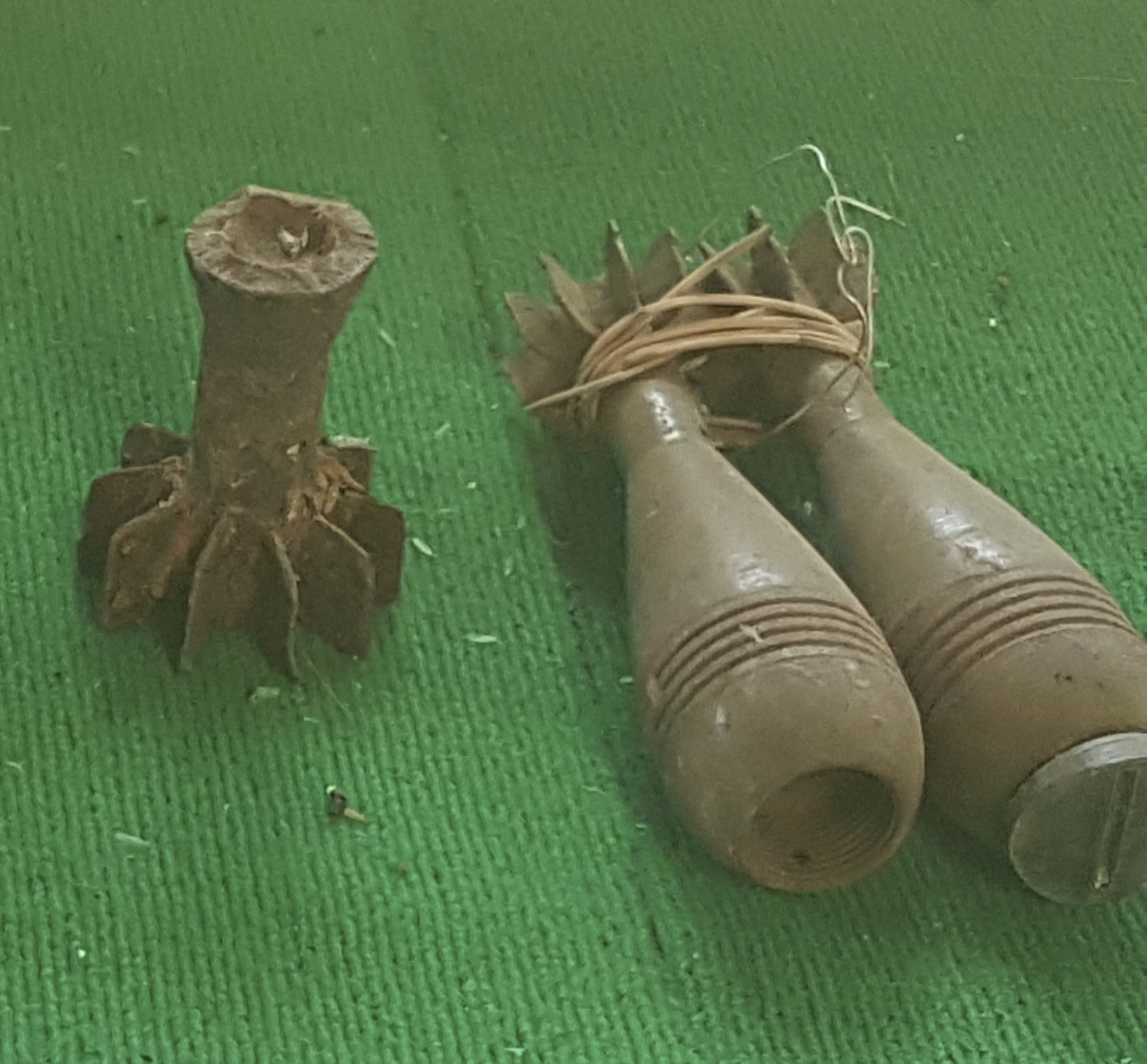} \\
\end{array}$
\end{center}
\caption{A collection of decommissioned UXO (including both mortar bomb shell casings and hand (fragmentation) grenades) that are housed at the Cambodia War Remnant Museum, with artefacts being characterised by different cyclic symmetry groups including ${\mathfrak C}_4$, ${\mathfrak C}_5$, ${\mathfrak C}_6$, ${\mathfrak C}_8$, ${\mathfrak C}_{10}$. }\label{fig:uxo}
\end{figure}

\begin{figure}
\begin{center}
$\begin{array}{cc}
\includegraphics[width=1in]{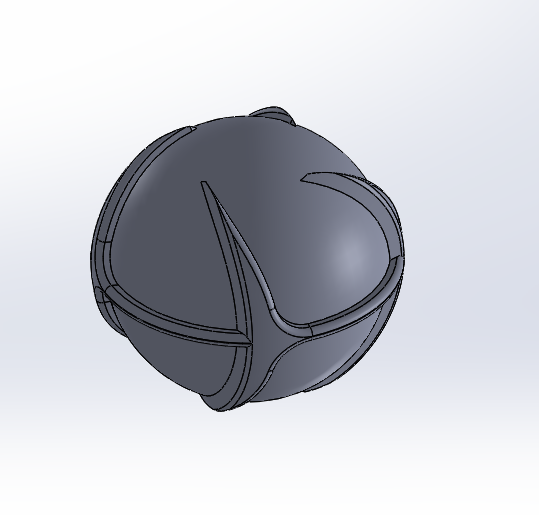} &
\includegraphics[width=2.2in]{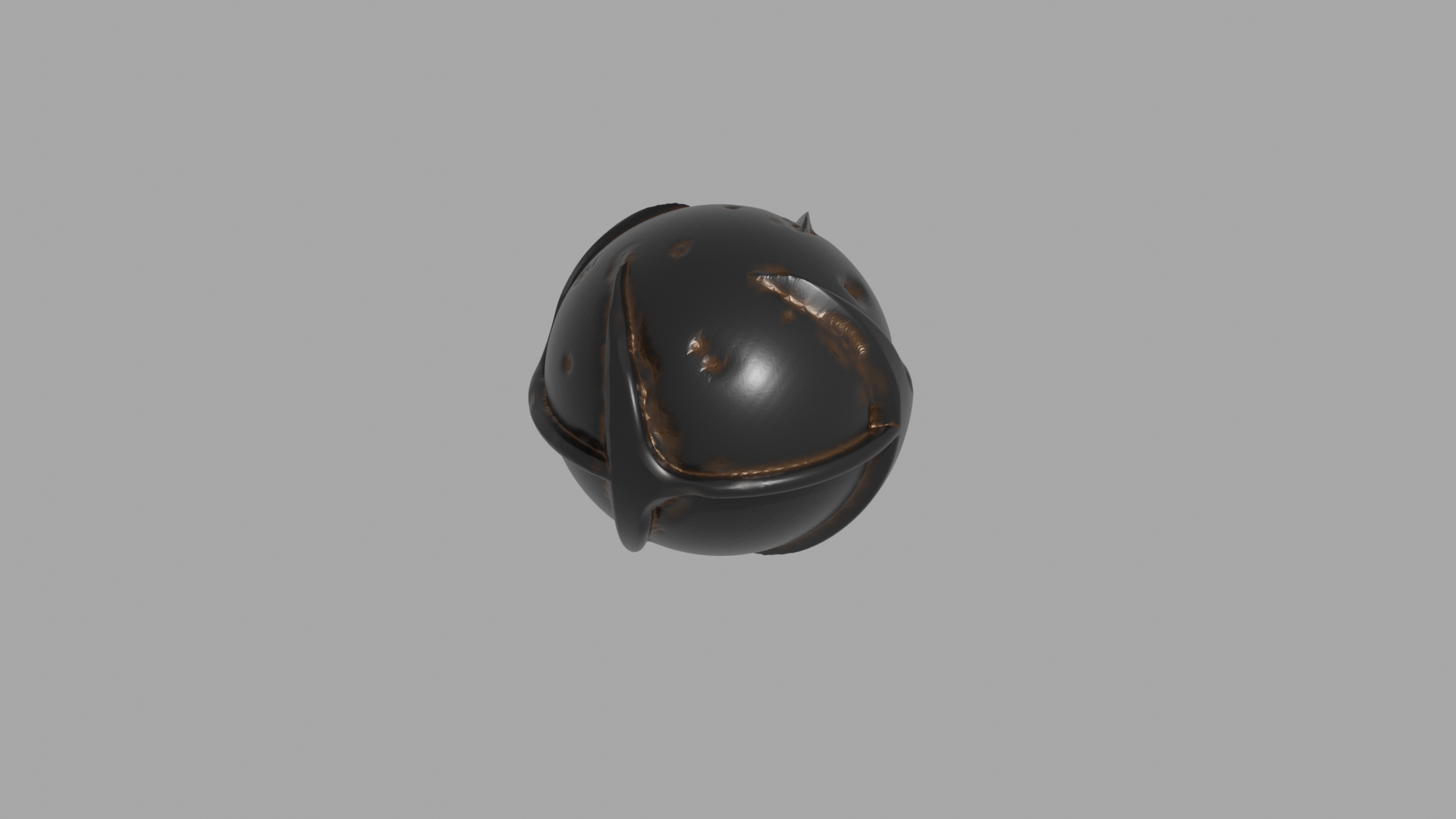} \\
\end{array}$
\end{center}
\caption{{A CAD drawing and associated rendered image of a BLU-61 submunition with a mirror symmetry as well as a 4--fold rotational symmetry and, hence, when suitably orientated, is characterised by the dihedral  ${\mathfrak D}_4$ group.}} \label{fig:blu61} 
\end{figure}

The results in the aforementioned tables described in Remark~\ref{tab:grptab} can be used to easily identify the invariant HGPT coefficients. For example, considering ${\vec S}_{pq}^{\mathfrak G}$ for $p=q=1$ and the  ${\mathfrak C}_2$ group, we can identify that 
\begin{align}
{\vec S}_{11}^{{\mathfrak C}_2} = \left \{  x_1 y_1, x_1 y_2 +x_2 y_1, x_2 y_2, x_3 y_3 \right \}, \nonumber
\end{align}
so this means that the non-zero independent HGPT coefficients are 
\begin{align}
M_{1,-1,1,-1}^{\text{H}}, M_{1,-1,1,0}^{\text{H}}=M_{1,0,1,-1}^{\text{H}} , M_{1,0,1,0}^{\text{H}}, M_{1,1,1,1}^{\text{H}}.
\end{align}
Similarly, for $S_{11}^{{\mathfrak C}_4}$ considered in Section~\ref{sect:example}, the non-zero independent HGPT coefficients are 
\begin{align}
M_{1,-1,1,-1}^{\text{H}}= M_{1,0,1,0}^{\text{H}}, M_{1,1,1,1}^{\text{H}}.
\end{align}

{Based on Section~\ref{sect:spharmoics} we can compare the number of independent coefficients of GPTs, HGPTs and CGPTs. Without assuming any object symmetries, a GPT, with orders $p$, $q$, has $(p+1)(p+2)(q+1)(q+2)/4$ independent coefficients for a standard monomial basis. However, the case of $p=q=1$ is  special as the monomial basis agrees (up to some possible scaling) with a harmonic basis in this case. If a harmonic polynomial basis is used, we have a symmetric HGPT with $(2p+1)(2q+1)$ coefficients (again without assuming object symmetries) and, in the case of $p=q$, by using Remark~\ref{remark:indhgpt},  the symmetry of the HGPT implies $ (2p+1)(p+1)$, which is fewer than $(2p+1)^2$. The CGPT, which uses a harmonic function basis, has the same number of independent coefficients as the HGPT. The advantage of the HGPT is that, if we know that the object is characterised by a particular symmetry group, we can use Algorithm~\ref{alg:subspace} to automatically  reduce the number of independent coefficients required for a HGPT. Once this has been achieved, the non-zero independent HGPT coefficients can be transformed to coefficients of another basis (for example the coefficients of CGPTs using (\ref{eqn:hgpt1})). Following this transformation, the reduced number of independent coefficients will remain the same. As an example, we compare in Table~\ref{tab:basisred} the number of independent coefficients of GPTs without assuming object symmetries (using a basis of standard monomials); CGPTs/HGPTs including tensor symmetries, but without assuming object symmetries (using a basis of harmonic functions and  polynomials, respectively), and the reduced dimension HGPTs, obtained by additionally taking account a cyclic ${\mathfrak C}_2$ group   symmetry group. Similar reductions also apply for other symmetry groups. }

\begin{table}
\begin{center}
\begin{tabular}{|c|c|c|c|c|}
\hline
$p $ & $q$ & GPT& Symmetric CGPT / HGPT  & Reduced dimension HGPT \\
& {} &  $M_{\alpha \beta}$  & $M^C_{qjpi}=M^C_{piqj}$, $M^H_{qjpi}=M^H_{piqj}$ & $\dim {\vec S}_{pq}^{{\mathfrak C}_2}$ \\
& {} &$|\alpha|=p$, $|\beta|=q$ &  $-p \le i \le p, -q \le j\le q$ &{} \\
\hline
1 & 1 & 9 (6) &  6  & 4\\
\hline
1 & 2 & 18 & 15  & 7 \\
\hline
1 & 3 & 30 & 21 & 11 \\ 
\hline
2 & 2 & 36 & 15 & 9 \\
\hline
\end{tabular}
\end{center}
\caption{{Comparison of the number of non-zero independent coefficients: GPTs without assuming object symmetries (using a standard monomial basis); CGPTs/HGPTs including tensor symmetries, but without assuming object symmetries (using a basis of harmonic functions and polynomials, respectively), and reduced dimension HGPTs, obtained by additionally taking account of the object symmetries, assuming the object is characterised by the ${\mathfrak C}_2$ group. The results are quoted for different orders $p$ and $q$. Note that for $p=q=1$ the standard monomial basis agrees with the harmonic polynomial basis  (up to some possible scaling) leading to the reduced dimension quoted in brackets in this case.} }\label{tab:basisred}
\end{table}

\begin{remark}
Related to the question of finding the equivalent class of objects that an (H)GPT of a given order describes,   raised in the introductory remarks of 
Section~\ref{sect:sym}, a further practical problem is to be able to  find objects (e.g UXO such as mortar bombs) with a given symmetry. In practical measurements,  the measured $V_{\text{sr}}$  will contain errors and unavoidable noise that are associated with measurements, which are not included in (\ref{eqn:Vsrharmonic}). Still further, buried objects are often dented and deformed as a result of falling from a height to the ground so that a hidden object's symmetries may only hold approximately in practice, as Figure~\ref{fig:uxo} illustrates.
With (machine learning) object classification in mind, the sets of non-zero independent HGPT coefficients for different symmetry groups offer a possible alternative  to object classification based on shape invariant descriptors proposed in~\cite{ammari3dinvgpt}. Here, it is envisaged that a classifier could be developed based on classifying objects according to their symmetry groups with the non-zero features being those independent HGPT coefficients for the symmetry group under consideration. 
Thus, offering the potential to detect objects of a certain cyclic (or dihedral) group up to the measurement error and errors in the object symmetry as well as contributing to understanding the additional information that higher order (H)GPTs provide.

\end{remark}

\begin{table}
\begin{center}
\begin{tabular}{|l|l|l|l|}
\hline
 $p$ &  $q$ & $\dim {\vec S}_{pq}^{{\mathfrak C}_2}$  & ${\vec S}_{pq}^{{\mathfrak C}_2}$  \\
\hline
1 & 1 & 4 & $\{x_1 y_1,$\\
{} & {} & {} & $x_2 y_1 + x_1 y_2,$\\
{} & {} & {}  & $x_2 y_2,$\\
{} & {} & {}  & $x_3y_3\}$ \\
\hline
1 & 2 & 7 & $\{x_1 x_3 y_1 + x_1 y_1 y_3,$\\
{} & {} & {}  &$x_2 x_3 y_1 + x_1 y_2 y_3,$\\
{} & {} & {}  &$x_1 x_3 y_2 + x_2 y_1 y_3,$\\
{} & {} & {}  &$x_2 x_3 y_2 + x_2 y_2 y_3,$\\
{} & {} & {}  &$x_3 (y_1^2 - y_2^2) + (x_1^2 - x_2^2) y_3,$\\
{} & {} & {}  &$(x_1^2 - x_3^2) y_3 + x_3 (y_1^2 - y_3^2),$\\
{} & {} & {}  &$x_3 y_1 y_2 + x_1 x_2 y_3\}$\\
\hline
1 & 3 & 11 & $\{(x_1^3 - 3 x_1 x_2^2) y_1 + x_1 (y_1^3 - 3 y_1 y_2^2),$ \\
{} & {} & {}  &$(-3 x_1^2 x_2 + x_2^3) y_1 + x_1 (-3 y_1^2 y_2 + y_2^3),$\\
{} & {} & {}  &$(x_1^3 - 3 x_1 x_3^2) y_1 + x_1 (y_1^3 - 3 y_1 y_3^2),$\\
{} & {} & {}  &$(x_2^3 - 3 x_2 x_3^2) y_1 + x_1 (y_2^3 - 3 y_2 y_3^2),$\\
{} & {} & {}  &$(x_1^3 - 3 x_1 x_2^2) y_2 + x_2 (y_1^3 - 3 y_1 y_2^2),$\\
{} & {} & {}  &$(-3 x_1^2 x_2 + x_2^3) y_2 + x_2 (-3 y_1^2 y_2 + y_2^3),$\\
{} & {} & {}  &$(x_1^3 - 3 x_1 x_3^2) y_2 + x_2 (y_1^3 - 3 y_1 y_3^2),$\\
{} & {} & {}  &$(x_2^3 - 3 x_2 x_3^2) y_2 + x_2 (y_2^3 - 3 y_2 y_3^2),$\\
{} & {} & {}  &$(-3 x_1^2 x_3 + x_3^3) y_3 + x_3 (-3 y_1^2 y_3 + y_3^3),$\\
{} & {} & {}  &$(-3 x_2^2 x_3 + x_3^3) y_3 + x_3 (-3 y_2^2 y_3 + y_3^3),$\\
{} & {} & {}  &$x_1 x_2 x_3 y_3 + x_3 y_1 y_2 y_3\}$\\
\hline
2 & 2& 9 &$ \{(x_1^2 - x_2^2) (y_1^2 - y_2^2),$\\
{} & {} & {}  &$(x_1^2 - x_3^2) (y_1^2 - y_2^2) + (x_1^2 - x_2^2) (y_1^2 - y_3^2),$ \\
{} & {} & {}  &$(x_1^2 - x_2^2) y_1 y_2 + x_1 x_2 (y_1^2 - y_2^2),$ \\
{} & {} & {}  &$(x_1^2 - x_3^2) (y_1^2 - y_3^2),$ \\
{} & {} & {}  &$(x_1^2 - x_3^2) y_1 y_2 + x_1 x_2 (y_1^2 - y_3^2),$ \\
{} & {} & {}  &$x_1 x_2 y_1 y_2,$ \\
{} & {} & {}  &$x_1 x_3 y_1 y_3,$ \\
{} & {} & {}  &$x_2 x_3 y_1 y_3 + x_1 x_3 y_2 y_3,$ \\
{} & {} & {}  &$x_2 x_3 y_2 y_3\}$ \\
\hline
\end{tabular}
\end{center}
\caption{The dimension and set of symmetric product harmonic polynomials ${\vec S}_{pq}^{\mathfrak G}$ fixed by the ${\mathfrak G}={\mathfrak C}_2$ group.} \label{tab:symc2}
\end{table}

\begin{table}
\begin{center}
\begin{tabular}{|l|l|l|l|}
\hline
 $p$ &  $q$ & $\dim {\vec S}_{pq}^{{\mathfrak C}_3}$  & ${\vec S}_{pq}^{{\mathfrak C}_3}$  \\
\hline
1 & 1 & 2 & $\{x_1 y_1 +  x_2 y_2,$\\
{} & {} & {} & $ x_3 y_3\}$ \\
\hline
1 & 2 & 5 &  $\{- (x_1 x_2 y_2 + x_2 y_1 y_2) + 1/2 ((x_1^2 - x_2^2) y_1 + x_1 (y_1^2 - y_2^2)),$ \\
{} & {} & {}  &  $ (x_1 x_2 y_1 + x_1 y_1 y_2) + 1/2 ((x_1^2 - x_2^2) y_2 + x_2 (y_1^2 - y_2^2)),$\\
{} & {} & {}  &  $ (x_1 x_3 y_1 + x_1 y_1 y_3) +  (x_2 x_3 y_2 + x_2 y_2 y_3),$ \\
{} & {} & {}  &  $- (x_1 x_3 y_2 + x_2 y_1 y_3) +  (x_2 x_3 y_1 + x_1 y_2 y_3),$ \\
{} & {} & {}  &  $ -(1/2) (x_3 (y_1^2 - y_2^2) + (x_1^2 - x_2^2) y_3) +  ((x_1^2 - x_3^2) y_3 + x_3 (y_1^2 - y_3^2))\}$ \\
\hline
1 & 3& 7 & $\{1/6 (-((x_1^3 - 3 x_1 x_2^2) y_1) - x_1 (y_1^3 - 3 y_1 y_2^2)) + $\\
 {} & {} & {}  & $1/6 (-((-3 x_1^2 x_2 + x_2^3) y_2) - x_2 (-3 y_1^2 y_2 + y_2^3)) + $\\
 {} & {} & {}  & $2/3 ((x_1^3 - 3 x_1 x_3^2) y_1 + x_1 (y_1^3 - 3 y_1 y_3^2)) + $\\
{} & {} & {}  & $ 2/3 ((x_2^3 - 3 x_2 x_3^2) y_2 + x_2 (y_2^3 - 3 y_2 y_3^2))$,\\
{} & {} & {}  & $x_1 x_2 x_3 y_2 + x_2 y_1 y_2 y_3 + $\\
{} & {} & {}  & $ 1/6 ((-3 x_1^2 x_3 + x_3^3) y_1 + x_1 (-3 y_1^2 y_3 + y_3^3)) + $\\
{} & {} & {}  & $ 1/6 (-((-3 x_2^2 x_3 + x_3^3) y_1) - x_1 (-3 y_2^2 y_3 + y_3^3))$,\\
{} & {} & {}  & $3/8 ((x_1^3 - 3 x_1 x_2^2) y_2 + x_2 (y_1^3 - 3 y_1 y_2^2)) - $\\
 {} & {} & {}  & $3/8 ((-3 x_1^2 x_2 + x_2^3) y_1 + x_1 (-3 y_1^2 y_2 + y_2^3)) - $\\
 {} & {} & {}  & $3/2 ((x_1^3 - 3 x_1 x_3^2) y_2 + x_2 (y_1^3 - 3 y_1 y_3^2)) + $\\
{} & {} & {}  & $ 3/2 ((x_2^3 - 3 x_2 x_3^2) y_1 + x_1 (y_2^3 - 3 y_2 y_3^2))$,\\
{} & {} & {}  & $x_1 x_2 x_3 y_1 + x_1 y_1 y_2 y_3 +$\\
{} & {} & {}  & $  1/6 (-((-3 x_1^2 x_3 + x_3^3) y_2) - x_2 (-3 y_1^2 y_3 + y_3^3)) + $\\
{} & {} & {}  & $ 1/6 ((-3 x_2^2 x_3 + x_3^3) y_2 + x_2 (-3 y_2^2 y_3 + y_3^3))$,\\
{} & {} & {}  & $x_3 (y_1^3 - 3 y_1 y_2^2) + (x_1^3 - 3 x_1 x_2^2) y_3$,\\
{} & {} & {}  & $x_3 (-3 y_1^2 y_2 + y_2^3) + (-3 x_1^2 x_2 + x_2^3) y_3$,\\
{} & {} & {}  & $1/2 ((-3 x_1^2 x_3 + x_3^3) y_3 + x_3 (-3 y_1^2 y_3 + y_3^3)) + $\\
{} & {} & {}  & $ 1/2 ((-3 x_2^2 x_3 + x_3^3) y_3 + x_3 (-3 y_2^2 y_3 + y_3^3))\}$\\
 \hline
  2 & 2& 5 & $\{12 x_1 x_2 y_1 y_2 + 3 (x_1^2 - x_2^2) (y_1^2 - y_2^2),$\\
   {} & {} & {}  & $1/2 (x_1 x_3 (y_1^2 - y_2^2) + (x_1^2 - x_2^2) y_1 y_3) -   (x_2 x_3 y_1 y_2 + x_1 x_2 y_2 y_3),$\\
     {} & {} & {}  & $ (x_1 x_3 y_1 y_2 + x_1 x_2 y_1 y_3) +  1/2 (x_2 x_3 (y_1^2 - y_2^2) + (x_1^2 - x_2^2) y_2 y_3),$\\
       {} & {} & {}  & $ x_1 x_2 y_1 y_2 + 3/4 (x_1^2 - x_2^2) (y_1^2 - y_2^2) + $\\
{} & {} & {}  & $ 2(x_1^2 - x_3^2) (y_1^2 - y_3^2) -  ((x_1^2 - x_3^2) (y_1^2 - y_2^2) + (x_1^2 - x_2^2) (y_1^2 - y_3^2)),$\\
         {} & {} & {}  & $ x_1 x_3 y_1 y_3 +  x_2 x_3 y_2 y_3\}$\\
         \hline
\end{tabular}
\end{center}
\caption{The dimension and set of symmetric product harmonic polynomials ${\vec S}_{pq}^{\mathfrak G}$ fixed by the ${\mathfrak G}={\mathfrak C}_3$ group.}  \label{tab:symc3}
\end{table}

\begin{table}
\begin{center}
\begin{tabular}{|l|l|l|l|}
\hline
 $p$ &  $q$ & $\dim {\vec S}_{pq}^{{\mathfrak C}_4}$  & ${\vec S}_{pq}^{{\mathfrak C}_4}$   \\
\hline
1 & 1 & 2 &  $\{x_1 y_1 + x_2 y_2,$\\
{} & {} & {}  &$x_3 y_3\}$\\
\hline
1 & 2 & 3 & $\{(x_1 x_3 y_1 + x_1 y_1 y_3) +  (x_2 x_3 y_2 + x_2 y_2 y_3),$\\
{} & {} & {}  &$- (x_1 x_3 y_2 + x_2 y_1 y_3) +  (x_2 x_3 y_1 + x_1 y_2 y_3),$\\
{} & {} & {}  &$-
 (x_3 (y_1^2 - y_2^2) + (x_1^2 - x_2^2) y_3) + 
 2 ((x_1^2 - x_3^2) y_3 + x_3 (y_1^2 - y_3^2))\}$\\
 \hline
 1 & 3 & 5 & $\{((x_1^3 - 3 x_1 x_2^2) y_1 + x_1 (y_1^3 - 3 y_1 y_2^2)) + $\\
 {} & {} & {}  &$  ((-3 x_1^2 x_2 + x_2^3) y_2 + x_2 (-3 y_1^2 y_2 + y_2^3)),$\\
  {} & {} & {}  &$ - ((x_1^3 - 3 x_1 x_2^2) y_2 + x_2 (y_1^3 - 3 y_1 y_2^2)) + $\\
{} & {} & {}  &$   ((-3 x_1^2 x_2 + x_2^3) y_1 + x_1 (-3 y_1^2 y_2 + y_2^3)),$\\
{} & {} & {}  &$  ((x_1^3 - 3 x_1 x_3^2) y_1 + x_1 (y_1^3 - 3 y_1 y_3^2)) + $\\
 {} & {} & {}  &$  ((x_2^3 - 3 x_2 x_3^2) y_2 + x_2 (y_2^3 - 3 y_2 y_3^2)),$\\
 {} & {} & {}  &$- ((x_1^3 - 3 x_1 x_3^2) y_2 + x_2 (y_1^3 - 3 y_1 y_3^2)) + $\\
 {} & {} & {}  &$   ((x_2^3 - 3 x_2 x_3^2) y_1 + x_1 (y_2^3 - 3 y_2 y_3^2)),$\\
 {} & {} & {}  &$ ((-3 x_1^2 x_3 + x_3^3) y_3 + x_3 (-3 y_1^2 y_3 + y_3^3)) + $\\
{} & {} & {}  &$   ((-3 x_2^2 x_3 + x_3^3) y_3 + x_3 (-3 y_2^2 y_3 + y_3^3)\}$\\
\hline
 2 & 2 & 5 & $\{(x_1^2 - x_2^2) (y_1^2 - y_2^2),$\\
 {} & {} & {}  &$  ((x_1^2 - x_2^2) y_1 y_2 + x_1 x_2 (y_1^2 - y_2^2)),$\\
 {} & {} & {}  &$(x_1^2 - x_2^2) (y_1^2 - y_2^2) + 2 (x_1^2 - x_3^2) (y_1^2 - y_3^2) -$\\ 
 {} & {} & {}  &$  ((x_1^2 - x_3^2) (y_1^2 - y_2^2) + (x_1^2 - x_2^2) (y_1^2 - y_3^2)),$\\
 {} & {} & {}  &$x_1 x_2 y_1 y_2,$\\
 {} & {} & {}  &$x_1 x_3 y_1 y_3 +  x_2 x_3 y_2 y_3\}$\\
 \hline
\end{tabular}
\end{center}
\caption{The dimension and set of symmetric product harmonic polynomials ${\vec S}_{pq}^{\mathfrak G}$ fixed by the ${\mathfrak G}={\mathfrak C}_4$ group.}  \label{tab:symc4}
\end{table}

\begin{table}
\begin{center}
\begin{tabular}{|l|l|l|l|}
\hline
 $p$ &  $q$ & $\dim {\vec S}_{pq}^{{\mathfrak C}_5}$  & ${\vec S}_{pq}^{{\mathfrak C}_5}$   \\
\hline
1 & 1 & 2 &$ \{ x_1 y_1 + x_2 y_2,$\\
{} & {} & {} &$  x_3 y_3\} $\\
\hline
1 & 2 & 2 & $\{ 1/2 (x_1 x_3 y_1 + x_1 y_1 y_3) + 1/2 (x_2 x_3 y_2 + x_2 y_2 y_3), $\\
{} & {} & {} &$-(5/2) (x_1 x_3 y_2 + x_2 y_1 y_3) + 5/2 (x_2 x_3 y_1 + x_1 y_2 y_3),$\\
{} & {} & {} &$-x_3 (y_1^2 - y_2^2) - (x_1^2 - x_2^2) y_3 + 
 2 ((x_1^2 - x_3^2) y_3 + x_3 (y_1^2 - y_3^2))\} $\\
 \hline
 1 & 3 & 3 &$\{ 1/6 (-((x_1^3 - 3 x_1 x_2^2) y_1) - x_1 (y_1^3 - 3 y_1 y_2^2)) +$\\
  {} & {} & {} &$ 1/6 (-((-3 x_1^2 x_2 + x_2^3) y_2) - x_2 (-3 y_1^2 y_2 + y_2^3)) + $\\
{} & {} & {} &$ 2/3 ((x_1^3 - 3 x_1 x_3^2) y_1 + x_1 (y_1^3 - 3 y_1 y_3^2)) + $\\
{} & {} & {} &$ 2/3 ((x_2^3 - 3 x_2 x_3^2) y_2 + x_2 (y_2^3 - 3 y_2 y_3^2)),$\\
{} & {} & {} &$ 5/8 ((x_1^3 - 3 x_1 x_2^2) y_2 + x_2 (y_1^3 - 3 y_1 y_2^2)) - $\\
{} & {} & {} &$ 5/8 ((-3 x_1^2 x_2 + x_2^3) y_1 + x_1 (-3 y_1^2 y_2 + y_2^3)) -  $\\
{} & {} & {} &$ 5/2 ((x_1^3 - 3 x_1 x_3^2) y_2 + x_2 (y_1^3 - 3 y_1 y_3^2)) +  $\\
{} & {} & {} &$5/2 ((x_2^3 - 3 x_2 x_3^2) y_1 + x_1 (y_2^3 - 3 y_2 y_3^2)),$\\
{} & {} & {} &$1/2 ((-3 x_1^2 x_3 + x_3^3) y_3 + x_3 (-3 y_1^2 y_3 + y_3^3)) + $\\
{} & {} & {} &$ 1/2 ((-3 x_2^2 x_3 + x_3^3) y_3 + x_3 (-3 y_2^2 y_3 + y_3^3))\}$\\
\hline
2 & 2 & 3 & $\{ 8/5 x_1 x_2 y_1 y_2 + 2/5 (x_1^2 - x_2^2) (y_1^2 - y_2^2)$,\\
{} & {} & {} &$ 8/7 x_1 x_2 y_1 y_2 + 6/7 (x_1^2 - x_2^2) (y_1^2 - y_2^2) +  16/7 (x_1^2 - x_3^2) (y_1^2 - y_3^2) - $\\
{} & {} & {} &$ 8/7 ((x_1^2 - x_3^2) (y_1^2 - y_2^2) + (x_1^2 - x_2^2) (y_1^2 - y_3^2)),$\\
{} & {} & {} &$x_1 x_3 y_1 y_3 + x_2 x_3 y_2 y_3\} $\\
\hline
\end{tabular}
\end{center}
\caption{The dimension and set of symmetric product harmonic polynomials ${\vec S}_{pq}^{\mathfrak G}$ fixed by the ${\mathfrak G}={\mathfrak C}_5$ group.}  \label{tab:symc5}
\end{table}

\begin{table}
\begin{center}
\begin{tabular}{|l|l|l|l|}
\hline
 $p$ &  $q$ & $\dim {\vec S}_{pq}^{{\mathfrak C}_6}$  & ${\vec S}_{pq}^{{\mathfrak C}_6}$ \\
\hline
1 & 1 & 2 &$ \{ x_1 y_1 + x_2 y_2,$\\
{} & {} & {} &$  x_3 y_3\} $\\
\hline
1 & 2 & 3 & $\{1/2 (x_1 x_3 y_1 + x_1 y_1 y_3) + 1/2 (x_2 x_3 y_2 + x_2 y_2 y_3),$\\
{} & {} & {} &$-3 (x_1 x_3 y_2 + x_2 y_1 y_3) + 3 (x_2 x_3 y_1 + x_1 y_2 y_3)$,\\
{} & {} & {} &$-x_3 (y_1^2 - y_2^2) - (x_1^2 - x_2^2) y_3 + 
 2 ((x_1^2 - x_3^2) y_3 + x_3 (y_1^2 - y_3^2))\} $\\
\hline
1 & 3 & 3 & $\{ 1/6 (-((x_1^3 - 3 x_1 x_2^2) y_1) - x_1 (y_1^3 - 3 y_1 y_2^2)) + $\\
 {} & {} & {} &$1/6 (-((-3 x_1^2 x_2 + x_2^3) y_2) - x_2 (-3 y_1^2 y_2 + y_2^3)) + $\\
{} & {} & {} &$ 2/3 ((x_1^3 - 3 x_1 x_3^2) y_1 + x_1 (y_1^3 - 3 y_1 y_3^2)) + $\\
{} & {} & {} &$ 2/3 ((x_2^3 - 3 x_2 x_3^2) y_2 + x_2 (y_2^3 - 3 y_2 y_3^2))$,\\
{} & {} & {} &$3/4 ((x_1^3 - 3 x_1 x_2^2) y_2 + x_2 (y_1^3 - 3 y_1 y_2^2)) - $\\
 {} & {} & {} &$3/4 ((-3 x_1^2 x_2 + x_2^3) y_1 + x_1 (-3 y_1^2 y_2 + y_2^3)) - $\\
 {} & {} & {} &$3 ((x_1^3 - 3 x_1 x_3^2) y_2 + x_2 (y_1^3 - 3 y_1 y_3^2)) + $\\
  {} & {} & {} &$3 ((x_2^3 - 3 x_2 x_3^2) y_1 + x_1 (y_2^3 - 3 y_2 y_3^2))$,\\
{} & {} & {} &$1/2 ((-3 x_1^2 x_3 + x_3^3) y_3 + x_3 (-3 y_1^2 y_3 + y_3^3)) + $\\
  {} & {} & {} &$1/2 ((-3 x_2^2 x_3 + x_3^3) y_3 + x_3 (-3 y_2^2 y_3 + y_3^3))\} $\\
\hline
2 & 2 & 3 & $\{8/5 x_1 x_2 y_1 y_2 + 2/5 (x_1^2 - x_2^2) (y_1^2 - y_2^2)$,\\
{} & {} & {} &$8/7 x_1 x_2 y_1 y_2 + 6/7 (x_1^2 - x_2^2) (y_1^2 - y_2^2) + $\\
{} & {} & {} & $16/7 (x_1^2 - x_3^2) (y_1^2 - y_3^2) - $\\
{} & {} & {} & $8/7 ((x_1^2 - x_3^2) (y_1^2 - y_2^2) + (x_1^2 - x_2^2) (y_1^2 - y_3^2))$,\\
{} & {} & {} &$x_1 x_3 y_1 y_3 + x_2 x_3 y_2 y_3\}$\\
\hline
\end{tabular}
\end{center}
\caption{The dimension and set of symmetric product harmonic polynomials ${\vec S}_{pq}^{\mathfrak G}$ fixed by the ${\mathfrak G}={\mathfrak C}_6$ group.}    \label{tab:symc6}
\end{table}

\section*{Acknowledgement}
Paul D. Ledger gratefully acknowledges the financial support received from EPSRC in the form of grants  EP/V049453/1 and EP/V009028/1. William R. B. Lionheart gratefully acknowledges the financial support received from EPSRC in the form of grants  EP/V049496/1 and EP/V009109/1 and would like to thank the Royal Society for the financial support received from a Royal Society Wolfson Research Merit Award and  a Royal Society  Global Challenges Research Fund grant CH160063. {The authors are grateful to Daniel Conniffe of The University of Manchester for creating the images shown in Figure~\ref{fig:blu61}}.

\appendix
\section{Results from Representation Theory} \label{sect:app}

For those unfamiliar with representation theory, some key results that are relevant for our work and accompanying references are provided below. 

Let ${\mathfrak G}$  be a finite group with representation $( {\vec \pi} , {\vec V})$ where ${\vec V}$ is a  vector space  of dimension $n$ with field $K$, $K$ having characteristic ${\mathbb R}$ or ${\mathbb C}$, and ${\vec \pi}$ is a homomorphism from ${\mathfrak G}$ to the group of invertible linear transformations of ${\vec V}$. We say ${\vec v} \in {\vec V} $ is fixed by ${\mathfrak G}$ if
\begin{equation}
 {\vec \pi } ({\mathbf G}) ({\vec v}) = {\vec v} \qquad \forall {\mathbf G} \in {\mathfrak G}. \nonumber
\end{equation}
The set of all elements fixed by ${\mathfrak G}$ is 
\begin{equation}
{\vec V}^{\mathfrak G} = \left \{  {\vec v} \in {\vec V} : {\vec \pi } ({\mathbf G})({\vec v})={\vec v} \ \forall {\mathbf G} \in {\mathfrak G}
\right \}, \nonumber
\end{equation}
which is a subspace of ${\vec V}$.

We denote by
\begin{equation}
{\mathbf M}_\pi : = \frac{1}{|{\mathfrak G}|} \sum_{{\mathbf G} \in {\mathfrak G}} {\vec \pi } ({\mathbf G}) , \label{eqn:avgrep}
\end{equation}
the average matrix in the representation.

From the definition (\ref{eqn:avgrep}) it follows that ${\vec \pi }({\mathbf G}){\mathbf M}_\pi = {\mathbf M}_\pi$ for all ${\mathbf G} \in {\mathfrak G}$, as
multiplying all the elements of a group by a fixed element simply reorders the element. Hence, ${\mathbf M}_\pi^2 ={\mathbf M}_\pi$ and ${\mathbf M}_\pi$ is a projection~\cite{lax}[pg. 30, pg. 84] on to ${\vec V}^{\mathfrak G}$. This clearly means that the eigenvalues of ${\mathbf M}_\pi$ can only be $1$ or $0$, but we can diagonalise the matrix ${\mathbf M}_\pi$ using the following change of basis. Let ${\vec v}_1, \ldots, {\vec v}_m$ be a basis for ${\vec V}^{\mathfrak G}$ so that ${\mathbf M}_\pi {\vec v}_i = {\vec v}_i$.  Now let ${\vec v}_{m+1}, \ldots, {\vec v}_n$ be a basis for the null space of ${\mathbf M}_\pi$ so that ${\mathbf M}_\pi {\vec v}_i= {\vec 0} $ for $i > m$. So the matrix of ${\mathbf M}_\pi$ in this new basis is
\begin{equation}
\tilde{\mathbf M}_\pi = \left (  \begin{array}{ll} {\mathbf I}_{m \times m}  & {\mathbf 0}_{m \times (n-m)} \\
 {\mathbf 0}_{ (n-m)\times m}  &  {\mathbf 0}_{(n-m )\times(n-m)}  \end{array} \right ). \label{eqn:avgreptil}
\end{equation}
Let ${\vec y}_i$, $i=1,\ldots,n$ be any basis for ${\mathbf V}$ and ${\vec w}_i  = {\mathbf M}_\pi {\vec y}_i$. As the range of ${\mathbf M}_\pi$ is exactly ${\vec V}^{\mathfrak G}$ it follows that 

\begin{lemma}
 ${\vec V}^{\mathfrak G}$ is spanned by  ${\vec w}_i$, $i=1,\ldots,n$.
\end{lemma}
Also, as the trace of matrix is invariant under a change of basis  $\Tr ({\mathbf M}_\pi) = \Tr (\tilde{\mathbf M}_\pi)$ ~\cite{lax}[Thm. 9, pg. 56] and we have
\begin{lemma}
\begin{equation}
m = \dim {\vec V}^{\mathfrak G} =\Tr (  {\mathbf M}_\pi  ). \nonumber
\end{equation}
\end{lemma}
The proof is immediate from taking the trace of (\ref{eqn:avgreptil}), but is a standard result in representation theory, see for example Fulton and Harris~\cite{fulton-harris}[Prop. 2.8, pg. 15-16].

\bibliographystyle{plainurl}
\bibliography{paperbib}

\end{document}